\documentclass[12pt,reqno]{amsart}

\usepackage[utf8]{inputenc}

\usepackage{amsmath} 
\usepackage{amsthm} 
\usepackage{amssymb} 
\usepackage{amscd} 

\usepackage{euscript}
\DeclareMathAlphabet{\mathpzc}{OT1}{pzc}{m}{it}
\DeclareMathAlphabet\euscr{T1}{qzc}{m}{n}

\usepackage{enumitem} 
\usepackage{mathtools} 
\usepackage{mathrsfs} 
\usepackage{hyperref} 
\hypersetup{
    linktocpage = true,
    colorlinks=true,
    linkcolor=red,
}
\usepackage{esint} 

\usepackage{graphicx} 
\usepackage{subcaption}

\usepackage{pgf,tikz,pgfplots}
\usetikzlibrary{arrows}
\usetikzlibrary{intersections}
\usetikzlibrary{fadings}

\usepackage[colorinlistoftodos]{todonotes} 
\newtheorem{theorem}{Theorem}[section]
\newtheorem*{theorem*}{Theorem}
\newtheorem{proposition}[theorem]{Proposition}
\newtheorem{lemma}[theorem]{Lemma}

\theoremstyle{definition}

\theoremstyle{remark}
\newtheorem{remark}[theorem]{Remark}
\newtheorem*{remark*}{Remark}

\newcommand{\con}[1]{\mathbb{#1}}
\newcommand{\R}{\con{R}} 
\newcommand{\Z}{\con{Z}} 
\newcommand{\Sph}{\con{S}} 
\renewcommand{\H}{\con{H}}
\newcommand{\ccal}{\mathscr{C}}
\newcommand{\ecal}{\mathcal{E}}
\newcommand{\ical}{\mathcal{I}}
\newcommand{\lcal}{\mathcal{L}}

\newcommand{\ocal}{\mathcal{O}}

\makeatletter
\newcommand{\leqnomode}{\tagsleft@true\let\veqno\@@leqno}
\newcommand{\reqnomode}{\tagsleft@false\let\veqno\@@eqno}
\makeatother

\newcommand{\norm}[1]{\left | \left |{#1} \right | \right |}
\newcommand{\seminorm}[1]{\left [ {#1} \right ] }

\newcommand{\s}{\gamma}
\newcommand{\fraclaplacian}{(-\Delta)^\s}

\newcommand{\Lip}{\mathrm{Lip}}
\renewcommand{\d}{\,\mathrm{d}} 
\newcommand{\bpar}[1]{\left ( {#1}\right )}
\newcommand{\setcond}[2]{\left \{ #1 \ : \ #2  \right \}}
\newcommand\evalat[1]{_{\mkern1.5mu\big\vert_{\scriptstyle #1}}}

\newcommand{\average}{\fint}
\newcommand\beqc[1]{\left\{\begin{array}{#1}}
\newcommand\eeqc{\end{array} \right.}

\def\PDEsystem{rcll}
\def\bmatrix{\begin{pmatrix}}
\def\ematrix{\end{pmatrix}}

\DeclareMathOperator{\dist}{dist}

\DeclareMathOperator{\sign}{sign}

\numberwithin{equation}{section}

\graphicspath {{Images/}}

\title[Semilinear integro-differential equations I]{Semilinear integro-differential equations, I: odd solutions with respect to the Simons cone}
\author{Juan-Carlos Felipe-Navarro}
\address{J.C. Felipe-Navarro:
Universitat Polit\`ecnica de Catalunya and BGSMath, Departament de Matem\`{a}tiques, Diagonal 647, 08028 Barcelona, Spain}
\email{juan.carlos.felipe@upc.edu}

\author{Tomás Sanz-Perela}
\address{T. Sanz-Perela:
Universitat Polit\`ecnica de Catalunya and BGSMath, Departament de Matem\`{a}tiques, Diagonal 647, 08028 Barcelona, Spain}
\email{tomas.sanz@upc.edu}

\thanks{Both authors acknowledge financial support from the Spanish Ministry of Economy and Competitiveness (MINECO), through the María de Maeztu Programme for Units of Excellence in R\&D (MDM-2014-0445-16-4 and MDM-2014-0445, respectively), are supported by MINECO grants MTM2014-52402-C3-1-P and MTM2017-84214-C2-1-P, are members of the Barcelona Graduate School of Mathematics (BGSMath), and are part of the Catalan research group 2017 SGR 01392.}

\keywords{Integro-differential semilinear equation, odd symmetry, Simons cone, maximum principle for odd functions, energy estimate, saddle-shaped solution}

\begin{document}

\begin{abstract}
This is the first of two papers concerning saddle-shaped solutions to the semilinear equation $L_K u = f(u)$ in $\R^{2m}$, where $L_K$ is a linear elliptic integro-differential operator and $f$ is of Allen-Cahn type. 

Saddle-shaped solutions are doubly radial, odd with respect to the Simons cone $\{(x', x'') \in \R^m \times \R^m \, : \, |x'| = |x''|\}$, and vanish only on this set. By the odd symmetry, $L_K$ coincides with a new operator $L_K^\ocal$ which acts on functions defined only on one side of the Simons cone, $\{|x'|>|x''|\}$, and that vanish on it. This operator $L_K^\ocal$, which corresponds to reflect a function oddly and then apply $L_K$, has a kernel on $\{|x'|>|x''|\}$ which is different from $K$. 

In this first paper, we characterize the kernels $K$ for which the new kernel is positive and therefore one can develop a theory on the saddle-shaped solution. The necessary and sufficient condition for this turns out to be that $K$ is radially symmetric and $\tau\mapsto K(\sqrt \tau)$ is a strictly convex function. 

Assuming this, we prove an energy estimate for doubly radial odd minimizers and the existence of saddle-shaped solution. In a subsequent article, part II, further qualitative properties of saddle-shaped solutions will be established, such as their asymptotic behavior, a maximum principle for the linearized operator, and their uniqueness.
\end{abstract}

\maketitle

\tableofcontents

\section{Introduction}
\label{Sec:Introduction}

In this paper we study solutions to the semilinear integro-differential equation
\begin{equation}
\label{Eq:NonlocalAllenCahn}
L_K u = f(u) \quad \textrm{ in } \R^{2m}
\end{equation}
which are odd with respect to the Simons cone --- defined in \eqref{Eq:SimonsCone}. The interest on these solutions, often called saddle-shaped solutions, is motivated by the nonlocal version of a conjecture by De Giorgi on the Allen-Cahn equation (see details below) with the aim of finding a counterexample in high dimensions. Moreover, this problem is related to the regularity theory of nonlocal minimal surfaces.

There are only three papers in the literature concerning saddle-shaped solutions to \eqref{Eq:NonlocalAllenCahn} with $L_K$ being the fractional Laplacian: \cite{Cinti-Saddle, Cinti-Saddle2} by Cinti and \cite{Felipe-Sanz-Perela:SaddleFractional} by the authors. In all of them the main tool is the extension problem. This paper, together with its second part  \cite{FelipeSanz-Perela:IntegroDifferentialII}, is the first one to study \eqref{Eq:NonlocalAllenCahn} without the extension. For this reason our arguments are purely nonlocal and hold for a more general class of kernels.

Equation \eqref{Eq:NonlocalAllenCahn} is driven by an integro-differential operator $L_K$ of the form
\begin{equation}
\label{Eq:DefOfLu}
L_Ku(x) = \int_{\R^n} \{u(x) - u(y)\} K(x-y)\d y,
\end{equation}
where the kernel $K$ satisfies
\begin{equation}
\label{Eq:Symmetry&IntegrabilityOfK}
K\geq 0\,, \quad K(z) = K(-z) \quad \textrm{ and } \quad \int_{\R^n} \min \left\{ |z|^2, 1 \right\} K(z) \d z < + \infty\,.
\end{equation}
The integral in \eqref{Eq:DefOfLu} has to be understood in the principal value sense. The most canonical example of such operators is the fractional Laplacian, defined for $\s\in(0,1)$ as
$$
\fraclaplacian u = c_{n, \s} \int_{\R^n} \dfrac{u(x) - u(y)}{|x-y|^{n + 2\s}}\d y\,,
$$
where $c_{n, \s}$ is a normalizing constant.

Recall that the fractional Laplacian has an associated extension problem (see \cite{CaffarelliSilvestre}) that allows the use of local arguments to deal with equations such as \eqref{Eq:NonlocalAllenCahn}. This is not the case for general operators $L_K$, and therefore some purely nonlocal techniques are developed along this work. 

Throughout the paper, we assume that $L_K$ is uniformly elliptic, that is,
\begin{equation}
\label{Eq:Ellipticity}
\lambda \dfrac{c_{n,\s}}{|z|^{n+2\s}} \leq K(z) \leq \Lambda \dfrac{c_{n,\s}}{|z|^{n+2\s}}\,, 
\end{equation}
where $\lambda$ and $\Lambda$ are two positive constants. This condition is frequently adopted since it yields Hölder regularity of solutions (see \cite{RosOton-Survey,SerraC2s+alphaRegularity}). The family of linear operators satisfying conditions \eqref{Eq:Symmetry&IntegrabilityOfK} and \eqref{Eq:Ellipticity} is the so-called $\lcal_0(n,\s,\lambda, \Lambda)$ ellipticity class. For short we will usually write $\lcal_0$ and we will make explicit the parameters only when needed. 

Moreover, for many  purposes we will need the operators to be invariant under rotations. This is equivalent to saying that the kernel is radially symmetric, $K(z) = K(|z|)$.

The Simons cone will be a central object along this paper. It is defined in $\R^{2m}$ by
\begin{equation}
\label{Eq:SimonsCone}
\mathscr{C} := \setcond{x = (x', x'') \in \R^m \times \R^m=\R^{2m}}{|x'| = |x''|}.
\end{equation}
This cone is of special importance in the theory of local and nonlocal minimal surfaces, and its variational properties are related to the conjecture of De Giorgi (see the end of this introduction for more details). Through the whole article we will use $\ocal$ and $\ical$ to denote each of the parts in which $\R^{2m}$ is divided by the cone $\ccal$:
$$
\ocal:= \setcond{x = (x', x'') \in \R^{2m}}{|x'| > |x''|} \textrm{ and } \,
\ical:= \setcond{x = (x', x'') \in \R^{2m}}{|x'| < |x''|}\!.
$$

Both $\ocal$ and $\ical$ belong to a family of sets in $\R^{2m}$ which are called of \emph{double revolution}. These are sets that are invariant under orthogonal transformations in the first $m$ variables, as well as under orthogonal transformations in the last $m$ variables. That is, $\Omega\subset \R^{2m}$ is a set of double revolution if $R\Omega = \Omega$ for every given transformation $R\in O(m)^2 = O(m) \times O(m)$, where  $O(m)$ is the orthogonal group of $\R^m$.

In this paper we deal with functions that are \emph{doubly radial}. These are functions $w:\R^{2m}  \to \R$ that only depend on the modulus of the first $m$ variables and on the modulus of the last $m$ ones, i.e., $w(x) = w(|x'|,|x''|)$. Equivalently, $w(Rx) = w(x)$ for every $R \in O(m)^2$.

In order to define oddness and evenness of functions with respect to the Simons cone, we consider the following isometry, which will play a significant role in this article:
\begin{equation}
\label{Eq:DefStar}
\begin{matrix}
(\cdot)^\star \colon & \R^{2m}= \R^{m}\times \R^{m}  &\to&  \R^{2m}= \R^{m}\times \R^{m}  \\
& x = (x',x'') &\mapsto & x^\star = (x'',x')\,.
\end{matrix}
\end{equation}
Note that this isometry is actually an involution that maps $\ocal$ into $\ical$ (and vice versa) and leaves the cone $\ccal$ invariant ---although not all points in $\ccal$ are fixed points of $(\cdot)^\star$. Taking into account this transformation, we say that a doubly radial function $w$ is \emph{odd with respect to the Simons cone} if $w(x) = -w(x^\star)$. Similarly, we say that a doubly radial function $w$ is \emph{even with respect to the Simons cone} if $w(x) = w(x^\star)$.

Regarding the doubly radial symmetry we define the following variables
$$
s := |x'| \quad \text{ and } \quad t:=|x''|\,.
$$
They are specially useful when dealing with the Laplacian in these coordinates, since
\begin{equation}
\label{Eq:Laplacian-st}
\Delta w = w_{ss} + w_{tt} + \frac{m-1}{s}w_s + \frac{m-1}{t}w_t
\end{equation}
becomes an expression suitable to work with. A similar formula appears in the case of the fractional Laplacian thanks to the local extension problem. Having a PDE in the two variables $(s,t)\in \R^2$ is useful to perform certain computations (see \cite{CabreTerraI, CabreTerraII,Cabre-Saddle, CabreRosOton-DoubleRev} for the local case and \cite{Cinti-Saddle, Cinti-Saddle2, Felipe-Sanz-Perela:SaddleFractional} for the fractional framework).

If we try to follow the same strategy by writing a rotation invariant operator $L_K$ in $(s,t)$ variables, the expression of the new operator is quite complex. Indeed, if $w:\R^{2m} \to \R$ is doubly radial and we define $\widetilde{w}(s,t) := w(s,0,...,0,t,0,...,0)$, it holds
$$ L_Kw(x) = \widetilde{L}_K \widetilde{w} (|x'|,|x''|)$$
with
\begin{equation}
\label{Eq:L_K-st}
\widetilde{L}_K \widetilde{w} (s,t) := \int_0^{+\infty}  \int_0^{+\infty} \sigma^{m-1} \tau^{m-1} \big(\widetilde{w}(s,t) - \widetilde{w}(\sigma, \tau)\big) J(s,t,\sigma, \tau)  \d \sigma\d \tau
\end{equation}
and
\begin{align*}
J(s,t,\sigma, \tau) &:= \int_{\Sph^{m-1}}  \int_{\Sph^{m-1}} K\Big( \sqrt{s^2+\sigma^2- 2 s \sigma \omega_1 + t^2 + \tau^2 - 2t \tau\tilde\omega_1}\Big) \d \omega \d \tilde\omega\,.
\end{align*}

Note that $\widetilde{L}_K$ is an integro-differential operator in $(0,+\infty)\times(0,+\infty)$, but the expression of its kernel is quite involved. Indeed, such an expression does not become simpler even when $L_K$ is the fractional Laplacian. In this case, the kernel $J$ involves hypergeometric functions of two variables, the so-called Appell functions (see Appendix~\ref{Sec:stcomputations} for more details on it), but this does not simplify computations.

Instead of working with the $(s,t)$ variables, we follow another approach that we find more clear and concise. It consists on rewriting the operator $L_K$ with a different kernel $\overline{K} : \R^{2m}\times \R^{2m} \to \R$ that is doubly radial with respect to its both arguments, but in such a way that it still acts on functions defined in $\R^{2m}$ ---and not in $(0,+\infty)^2$. As it is explained in detail in Section~\ref{Sec:OperatorOddF}, if $K$ is a radially symmetric kernel, then we can write $L_K$ acting on a doubly radial function $w$ as
\begin{equation}
\label{Eq:L_KWithKbar}
L_K w(x) = \int_{\R^{2m}} \{w(x) - w(y)\} \overline{K}(x,y) \d y\,,
\end{equation}
where $\overline{K} : \R^{2m}\times \R^{2m} \to \R$ is doubly radial in both arguments and is defined by
\begin{equation}
\label{Eq:KbarDef'}
\overline{K}(x,y) := \average_{O(m)^2} K(|Rx - y|)\d R\,.
\end{equation}
Here, $\d R$ denotes integration with respect to the Haar measure on $O(m)^2$ (see Section~\ref{Sec:OperatorOddF} for the details).

This new expression \eqref{Eq:L_KWithKbar} has some advantages compared with \eqref{Eq:L_K-st}. First, the computations in this new setting are shorter and more transparent than the analogous ones using $(s,t)$ variables. This also makes the notation more concise. Furthermore we avoid some issues of the $(s,t)$ variables such as the special treatment of the set $\{st=0\}$. Although in this paper we do not work in $(s,t)$ variables, we include an appendix at the end of the article with some computations using them (see Appendix~\ref{Sec:stcomputations}). We think that this could be useful in future works.

Once we have rewritten $L_K$ with a doubly radial kernel $\overline{K}$, as in \eqref{Eq:L_KWithKbar}, we shall find a suitable expression of the operator when acting on odd functions with respect to the Simons cone. Note that such functions are defined by their values in $\ocal$ and therefore we want to rewrite $L_K$ taking this into account. To this purpose, we define the new operator
\begin{equation}
\label{Eq:OperatorOddF}
\begin{split}
L_K^\ocal w (x)  &:= \int_{\ocal} \{w(x) - w(y) \} \overline{K}(x, y) \d y +  \int_{\ocal} \{w(x) + w(y) \} \overline{K}(x, y^\star) \d y \\
&= \int_{\ocal} \{w(x) - w(y) \} \{\overline{K}(x, y) - \overline{K}(x, y^\star)  \} \d y +  2 w(x) \int_{\ocal} \overline{K}(x, y^\star) \d y \,,
\end{split}
\end{equation}
where $(\cdot)^\star$ is defined in \eqref{Eq:DefStar}. As we show in Section~\ref{Sec:OperatorOddF}, $L_K^\ocal$ acting on a doubly radial function $w:\ocal \to \R$ coincides with $L_K$ acting on the odd extension of $w$ with respect to the Simons cone.

Our first main result concerns necessary and sufficient conditions on the original kernel $K$ for this operator to have a positive kernel.  As we will stress through this paper, and also in the forthcoming work \cite{FelipeSanz-Perela:IntegroDifferentialII}, the positivity of the kernel in \eqref{Eq:OperatorOddF} is crucial in order to develop a theory on the saddle-shaped solution. In particular, under this assumption a maximum principle for doubly radial odd functions will hold (see Proposition~\ref{Prop:MaximumPrincipleForOddFunctions} below).

\begin{theorem}
	\label{Th:SufficientNecessaryConditions}
	Let $K:(0,+\infty) \to (0,+\infty)$ and consider the radially symmetric kernel $K(|x-y|)$ in $\R^{2m}$. Define $\overline{K} : \R^{2m}\times \R^{2m} \to \R$ by \eqref{Eq:KbarDef'}.
	
	If 
	\begin{equation}
	\label{Eq:SqrtConvex}	
	K(\sqrt{\tau}) \text{ is a strictly convex function of }\tau\,,
	\end{equation}
	then $L_K$ has a positive kernel in $\ocal$ when acting on doubly radial functions which are odd with respect to the Simons cone $\ccal$. More precisely, it holds
	\begin{equation}
	\label{Eq:KernelInequality}
	\overline{K}(x,y) > \overline{K}(x, y^\star) \quad \text{ for every }x,y \in \ocal\,.
	\end{equation}
	
	In addition, if $K\in C^2((0,+\infty))$, then \eqref{Eq:SqrtConvex} is not only a sufficient condition for \eqref{Eq:KernelInequality} to hold, but also a necessary one.
\end{theorem}

This theorem is proved in Section~\ref{Sec:OperatorOddF} (see Propositions~\ref{Prop:KernelInequalitySufficientCondition} and \ref{Prop:KernelInequalityNecessaryCondition}). Its proof is based on breaking the integral defining $\overline{K}$ in four clever regions ---see \eqref{Eq:DefQ}--- that allow to compare the integrands for $y\in \ocal$ and for its reflected $y^*\in \ical$. We will use a result on convex functions proved in Appendix~\ref{Sec:AuxiliaryResults} (Proposition~\ref{Prop:EquivalenceK(sqrt)Convex<->Inequality}). In the previous statement, by strict convexity in \eqref{Eq:SqrtConvex} we mean that
$$
K(\sqrt{\tau_1}) + K(\sqrt{\tau_2}) > 2 K(\sqrt{(\tau_1 + \tau_2)/2})
$$
for every $\tau_1$, $\tau_2 \in (0,+\infty)$.

In \cite{JarohsWeth}, Jarohs and Weth study solutions to general integro-differential equations which are odd with respect to a hyperplane. Here the natural sufficient condition on $K$ to have a positive kernel when acting on odd functions is that $K$ is decreasing in the orthogonal direction to the hyperplane. That this suffices is readily deduced after making a change of variables given by the symmetry with respect to such hyperplane. In our case, since we deal with a more complex symmetry, the kernel $K$ is required to satisfy further assumptions than just monotonicity. Moreover, the proof of Theorem~\ref{Th:SufficientNecessaryConditions} is quite involved and requires a finer argument. Indeed, if we simply make the change $y \mapsto y^\star$ in $\eqref{Eq:DefOfLu}$, following \cite{JarohsWeth}, we should prove that $K(|x-y|) > K(|x-y^\star|)$ for every $x$ and $y$ in $\ocal$, but this is false even in the easiest case $L_K = \fraclaplacian$ and $2m=2$. Instead, if we write $L_K$ in the form \eqref{Eq:L_KWithKbar} with the kernel $\overline{K}$, the analogous positivity condition \eqref{Eq:KernelInequality} holds if we assume $K(\sqrt{\cdot})$ to be convex. Here the use of the $(s,t)$ variables would not simplify the proof of Theorem~\ref{Th:SufficientNecessaryConditions}. As mentioned in Appendix~\ref{Sec:stcomputations}, an analogous result can be established for the kernel $J$ in \eqref{Eq:L_K-st}, but its proof presents exactly the same difficulties as the one for $\overline{K}$.

The first direct consequence of the positivity condition \eqref{Eq:KernelInequality} is the following maximum principle.

\begin{proposition}[Maximum principle for odd functions with respect to $\ccal$]
	\label{Prop:MaximumPrincipleForOddFunctions} Let $\Omega \subset \ocal$ be an open set and let $L_K$ be an integro-differential operator with a radially symmetric kernel $K$ satisfying the positivity condition \eqref{Eq:KernelInequality}.  Let $u\in C^{\alpha}(\Omega)\cap C(\overline{\Omega})\cap L^\infty(\R^{2m})$, with $\alpha > 2\s$, be a doubly radial function which is odd with respect to the Simons cone. 
	
	\begin{enumerate}[label=(\roman{*})]
		\item  (Weak maximum principle)
		Assume that
		$$
		\beqc{\PDEsystem}
		L_K u + c(x) u & \geq & 0 & \text{ in } \Omega\,,\\
		u & \geq & 0 & \text{ in } \ocal \setminus \Omega\,,
		\eeqc
		$$
		with $c \geq 0$, and that either
		$$
		\Omega \text{ is bounded} \quad \text{ or } \liminf_{x \in \ocal,\,|x|\to +\infty} u(x) \geq 0\,.
		$$
		Then, $u \geq 0$ in $\Omega$.
		
		\item (Strong maximum principle)  
		Assume that $L_K u + c(x) u\geq 0$ in $\Omega$, with $c$ any continuous function, and that $u\geq 0$ in $\ocal$. Then, either $u\equiv 0$ in $\ocal$ or $u > 0$ in $\Omega$.
	\end{enumerate} 
\end{proposition}

This statement differs from the usual maximum principle for $L_K$ in the fact that we only assume that $u$ is nonpositive in $\ocal\setminus \Omega$, instead of in $\R^{2m}\setminus \Omega$ (an assumption that makes no sense for odd functions). This form of maximum principle is analogous to the ones in \cite{ChenLiLi, JarohsWeth}, where similar statements are considered for functions that are odd with respect to a hyperplane.

Since in this paper we will always consider doubly radial functions $u$ which are odd with respect to the Simons cone, $L_K u =L_K^\ocal u$ in $\ocal$. Thus, to simplify the notation we will always write $L_K$ for $L_K^\ocal$. To mean that Proposition~\ref{Prop:MaximumPrincipleForOddFunctions} holds, we will say that $L_K$ has a maximum principle in $\ocal$ when acting on doubly radial odd functions.

Let us now turn to the variational problem from which equation \eqref{Eq:NonlocalAllenCahn} arises. As it is well known, \eqref{Eq:NonlocalAllenCahn} is the Euler-Lagrange equation associated to the energy functional
\begin{equation}
\label{Eq:Energy}
\begin{split}
\ecal(w, \Omega) &:= 
\dfrac{1}{4} \left \{ \int_\Omega \int_\Omega |w(x) - w(y)|^2 K(x-y) \d x \d y \right. \qquad \qquad \\
& \quad \quad \quad +\left. 2 \int_\Omega \int_{\R^{2m} \setminus \Omega} |w(x) - w(y)|^2 K(x-y) \d x \d y \right \} + \int_{\Omega} G(w) \d x \,,
\end{split}
\end{equation}
where $G$ a $C^2$ function satisfying $G' = -f$. In this paper, we assume the following conditions on $G$:
\begin{equation}
\label{Eq:HipothesesG}
G \textrm{ is even and } G\geq G(\pm 1 )=0 \textrm{ in } \R\,.
\end{equation}
Note that the previous conditions on $G$ yield that $f$ is a $C^1$ odd function with $f(0)=f(\pm 1)=0$. In some cases, as in Theorem~\ref{Th:Existence} below, we will further assume that $G(0)>0$. In such situation, equation \eqref{Eq:NonlocalAllenCahn} can be seen as a model for phase transitions. The Allen-Cahn nonlinearity, $f(u) = u-u^3$, is the most typical example.

Using the same type of arguments as for the operator $L_K$, we can rewrite the energy of doubly radial odd functions with a suitable new expression that involves the kernel 
$$\overline{K}(x,y)-\overline{K}(x,y^\star)>0$$
and that only takes into account the values of the functions in $\ocal$. This will be extremely useful in many computations and estimates involving the nonlocal energy $\ecal$ (see Sections~\ref{Sec:EnergyForOddF} and \ref{Sec:EnergyEstimate}). To write this new expression, we introduce the following notation.  For $A$, $B\subset \ocal$, two sets of double revolution, we define
\begin{equation*}
\begin{split}
I_w(A,B) := 2\int_A  \int_B  \ |w(x)-w(y)|^2 \left\{ \overline{K}(x,y) - \overline{K}(x,y^\star) \right\} \d x \d y  \\
+\, 4 \int_A  \int_B  \left\{w^2(x)+w^2(y)\right\} \overline{K}(x,y^\star) \d x \d y\,.
\end{split}
\end{equation*}
Then, as proved in Section~\ref{Sec:EnergyForOddF} (see Lemma~\ref{Lemma:ShortExpressionEnergy}), we can rewrite the energy of a doubly radial odd function $w$ as
\begin{equation}
\label{Eq:ShortExpressionEnergyIntro}
\ecal(w, \Omega) = \frac{1}{4} \big \{I_w(\Omega\cap\ocal,\Omega\cap\ocal) +  2I_w(\Omega\cap\ocal,\ocal\setminus\Omega) \big \} + 2\int_{\Omega\cap \ocal} G(w) \d x \,.
\end{equation}

Thanks to this new expression for the energy, we are able to establish the second main result of this paper. It is the following energy estimate for doubly radial odd minimizers of $\ecal$. To define such minimizers properly, we denote by $\widetilde{\H}^K_{0, \mathrm{odd}}(B_R)$ the space of doubly radial odd functions that vanish outside $B_R$ and for which the energy $\ecal$ is well defined (see Section~\ref{Sec:EnergyForOddF} for the precise definition). We say that $u\in \widetilde{\H}^K_{0, \mathrm{odd}}(B_R)$ is a doubly radial odd minimizer of $\ecal$ in $B_R$ if
$$
\ecal(u,B_R) \leq \ecal (w,B_R)
$$
for every $w\in \widetilde{\H}^K_{0, \mathrm{odd}}(B_R)$. 

\begin{theorem}
	\label{Th:EnergyEstimate} 
	Let $K$ be a radially symmetric kernel satisfying the convexity assumption \eqref{Eq:SqrtConvex}\footnote{In this theorem, as well as in Theorem~\ref{Th:Existence}, we assume \eqref{Eq:SqrtConvex} instead of the positivity condition \eqref{Eq:KernelInequality} ---recall that for $C^2$ kernels they are equivalent.  The reason for this is that in the proofs we will make use of some estimates that require the Lipschitz regularity of the kernel $K$ (see Remark~\ref{Remark:InteriorRegularity} below). Such regularity for $K$ holds if \eqref{Eq:SqrtConvex} is satisfied, but it is not clear if this happens, when $K \notin C^2$, assuming \eqref{Eq:KernelInequality} instead.} and such that $L_K\in \lcal_0(2m, \s, \lambda, \Lambda)$. Assume that $G$ is a potential satisfying \eqref{Eq:HipothesesG}. Let $S\geq2$ and let $u\in \widetilde{\H}^K_{0, \mathrm{odd}}(B_R)$ be a doubly radial odd minimizer of $\ecal$ in $B_R$, with $R>S+4$. 
	
	Then
	\begin{equation}
	\label{Eq:EnergyEstimate} 
	\ecal (u,B_S) \leq 
	\begin{cases}
	C \ S^{2m-2\s}\ \ \ \ &\textrm{if } \ \ \s\in(0,1/2),\\
	C \ S^{2m-1} \log S\ \ \ \ &\textrm{if } \ \ \s=1/2,\\
	C \ S^{2m-1}\ \ \ \ &\textrm{if } \ \ \s\in(1/2,1),\\
	\end{cases}
	\end{equation}
	where $C$ is a positive constant depending only on $m$, $\s$, $\lambda$, $\Lambda$, and $\norm{G}_{C^2([-1,1])}$.
\end{theorem}

In the proof of this result, a first basic ingredient is that $-1\leq u\leq 1$, as provided by Lemma~\ref{Lemma:DecreaseEnergy}. This information, $|u|\leq 1$, is also of importance for a solution of an Allen-Cahn equation, as in the existence Theorem~\ref{Th:Existence} below. That $|u|\leq 1$ is proved with a variational cutting argument: cutting above $1$ and below $-1$ reduces de energy. We believe that this property requires $\overline{K}(x,y)- \overline{K}(x,y^\star)$ to be nonnegative. In addition, the proof of Lemma~\ref{Lemma:DecreaseEnergy} is a priori not simple since it involves a nonlocal energy of functions with symmetries. We succeeded to greatly simplify the computations by writing the energy as in \eqref{Eq:ShortExpressionEnergy}, obtaining a short proof.

Note that Theorem~\ref{Th:EnergyEstimate} does not follow from the energy estimate for general minimizers stated in \cite{SavinValdinoci-EnergyEstimate} by Savin and Valdinoci. The minimizers that they consider do not have any type of symmetry. In our case, the function $u$ in the previous statement minimizes the energy in a smaller class of functions and the result in \cite{SavinValdinoci-EnergyEstimate} cannot be applied. Nevertheless, we are able to adapt the arguments of Savin and Valdinoci to our setting.  The strategy they follow is to compare the energy of $u$ with the one of a suitable competitor which is constructed by taking the minimum between $u$ and a radially symmetric auxiliary function ---see \eqref{Eq:DefOfPhiS} below. Such competitor is not permitted in our case, since it is not odd with respect to the Simons cone. Nevertheless, we show in Section~\ref{Sec:EnergyEstimate} how to modify the auxiliary functions of \cite{SavinValdinoci-EnergyEstimate} to carry out the same type of arguments. The assumption \eqref{Eq:SqrtConvex} will be crucial to guarantee that $0\leq u \leq 1$ in $\ocal$.

The particular result of Theorem~\ref{Th:EnergyEstimate} for the fractional Laplacian has been proved by Cabré and Cinti \cite{CabreCinti-EnergyHalfL} in the case of the half-Laplacian, and extended to all the powers $0<\s<1$ by Cinti \cite{Cinti-Saddle2} (see \cite{CabreCinti-SharpEnergy} for an extension to non-doubly radial minimizers). These papers use the local extension problem and therefore their proofs cannot be extended to general operators like $L_K$. Our proof, following \cite{SavinValdinoci-EnergyEstimate}, overcomes this issue.

As an application of the previous results, we prove, by using standard variational methods, the existence of saddle-shaped solution to \eqref{Eq:NonlocalAllenCahn} when $f$ is of Allen-Cahn type. We say that a bounded solution $u$ to \eqref{Eq:NonlocalAllenCahn} is a \emph{saddle-shaped} solution if $u$ is doubly radial, odd with respect to the Simons cone, and positive in $\ocal$. 

\begin{theorem}[Existence of saddle-shaped solution]
	\label{Th:Existence}
	Let $G$ satisfy \eqref{Eq:HipothesesG}, $G(0)>0$, and let $f=-G'$. Let $K$ be a radially symmetric kernel satisfying the convexity assumption \eqref{Eq:SqrtConvex} and such that $L_K\in \lcal_0(2m, \s, \lambda, \Lambda)$. 
	
	Then, for every even dimension $2m \geq 2$, there exists a saddle-shaped solution $u$ to \eqref{Eq:NonlocalAllenCahn}. In addition, $u$ satisfies $|u|<1$ in $\R^{2m}$.
\end{theorem}

We are interested in the study of this type of solutions since they are relevant in connection with a famous conjecture for the (classical) Allen-Cahn equation raised by De Giorgi, that reads as follows. Let $u$ be a bounded monotone (in some direction) solution to $-\Delta u = u - u^3$ in $\R^n$, then, if $n \leq 8$, $u$ depends only on one Euclidean variable, that is, all its level sets are hyperplanes. This conjecture is not completely closed (see \cite{FarinaValdinoci-DeGiorgi} and references therein) but a counterexample in dimension $n=9$ was build in \cite{delPinoKowalczykWei} by using the so-called gluing method. Saddle-shaped solutions are natural objects to build a counterexample in a simpler way, as explained next. On the one hand, Jerison and Monneau \cite{JerisonMonneau} showed that a counterexample to the conjecture of De Giorgi in $\R^{n+1}$ can be constructed with a rather natural procedure if there exists a global minimizer of $-\Delta u = f(u)$ in $\R^n$ which is bounded and even with respect to each coordinate, but is not one-dimensional. On the other hand, by the $\Gamma$-converge results from Modica and Mortola (see \cite{Modica,ModicaMortola}) and the fact that the Simons cone is the simplest nonplanar minimizing minimal surface, saddle-shaped solutions are expected to be global minimizers of the Allen-Cahn equation in dimensions $2m\geq 8$ (this is still an open problem).

Similar facts happen in the nonlocal setting (see the introduction of \cite{Felipe-Sanz-Perela:SaddleFractional} for further details). For this reason, saddle-shaped solutions are of interest in the study of the nonlocal version of the conjecture of De Giorgi for equation \eqref{Eq:NonlocalAllenCahn}.

Saddle-shaped solutions to the local Allen-Cahn equation involving the Laplacian were studied in \cite{DangFifePeletier, Schatzman, CabreTerraI,CabreTerraII, Cabre-Saddle}. In these works, it is established the existence, uniqueness, and some qualitative properties of this type of solutions, such as their instability when $2m\leq 6$ and their stability if $2m\geq 14$. Stability in dimensions $8, 10$, and $12$ is still an open problem, as well as minimality in dimensions $2m\geq 8$.

In the fractional framework, there are only three works concerning saddle-shaped solutions to the equation $\fraclaplacian u = f(u)$. In \cite{Cinti-Saddle,Cinti-Saddle2}, Cinti proved the existence of saddle-shaped solution as well as some qualitative properties such as their asymptotic behavior, some monotonicity properties, and their instability in low dimensions. In a previous paper by the authors \cite{Felipe-Sanz-Perela:SaddleFractional}, further properties of these solutions have been established, the main ones being uniqueness and, when $2m\geq 14$, stability. The present paper together with its second part \cite{FelipeSanz-Perela:IntegroDifferentialII} are the first ones studying saddle-shaped solutions for general integro-differential equations of the form \eqref{Eq:NonlocalAllenCahn}. In the three previous papers \cite{Cinti-Saddle, Cinti-Saddle2, Felipe-Sanz-Perela:SaddleFractional}, the main tool used is the extension problem for the fractional Laplacian (see \cite{CaffarelliSilvestre}). As mentioned, this technique cannot be carried out for general integro-differential operators different from the fractional Laplacian. Therefore, some purely nonlocal techniques are developed through both papers.

In the forthcoming paper \cite{FelipeSanz-Perela:IntegroDifferentialII}, we study saddle-shaped solutions to \eqref{Eq:NonlocalAllenCahn} in more detail taking advantage of the setting for odd functions built in the present article. We give an alternative proof for the existence of a saddle-shaped solution by using monotone iteration and maximum principle techniques. As in the proof of Theorem~\ref{Th:Existence}, the assumtion \eqref{Eq:KernelInequality} is crucial. Furthermore, we prove the asymptotic behaviour of this type of solutions by using some symmetry and Liouville type results for general integro-differential operators that we establish in the same paper. Finally, we also show in \cite{FelipeSanz-Perela:IntegroDifferentialII} the uniqueness of the saddle-shaped solution through a maximum principle for the linearized operator, which we also prove in that article.

Let us make some final remarks on the minimality and stability properties of the Simons cone. Recall that, in the classical theory of minimal surfaces, it is well known that the Simons cone has zero mean curvature at every point $x\in \ccal \setminus \{0\}$, in all even dimensions, and it is a minimizer of the perimeter functional when $2m\geq 8$. Concerning the nonlocal setting, $\ccal$ has also zero nonlocal mean curvature in all even dimensions, although it is not known if it is a minimizer of the nonlocal perimeter in any dimension. If $2m=2$ it cannot be a minimizer since in \cite{SavinValdinoci-Cones} it is proven that all minimizing nonlocal minimal cones in $\R^2$ are flat. In higher dimensions, the only available results appear in \cite{DaviladelPinoWei, Felipe-Sanz-Perela:SaddleFractional} but concern stability, a weaker property than minimality. In \cite{DaviladelPinoWei},  Dávila, del Pino, and Wei characterize the stability of Lawson cones ---a more general class of cones that includes $\ccal$--- through an inequality involving only two hypergeometric constants which depend only on $\s$ and the dimension $n$. This inequality is checked numerically in \cite{DaviladelPinoWei}, finding that, in dimensions $n \leq 6$ and for $\s$ close to zero, no Lawson cone with zero nonlocal mean curvature is stable. Numerics also shows that all Lawson cones in dimension $7$ are stable if $\s$ is close to zero. These results for small $\s$ fit with the general belief that, in the fractional setting, the Simons cone should be stable (and even a minimizer) in dimensions $2m \geq 8$ (as in the local case), probably for all $\s\in(0,1/2)$, though this is still an open problem. In \cite{Felipe-Sanz-Perela:SaddleFractional}, we proved, by using the saddle-shaped solution to the fractional Allen-Cahn equation and a $\Gamma$-convergence result of \cite{CabreCintiSerra-Stable}, that the Simons cone is a stable $(2\s)$-minimal cone in dimensions $2m\geq 14$. To the best of our knowledge, this is the first analytical proof of a stability result for the Simons cone in any dimension.

This paper is organized as follows. Section~\ref{Sec:OperatorOddF} is devoted to study the operator $L_K$ acting on doubly radial odd functions. We deduce the expression of the kernel $\overline{K}$ and rewrite the operator acting on doubly radial odd functions, finding the expression \eqref{Eq:OperatorOddF}. We also show Theorem~\ref{Th:SufficientNecessaryConditions} and Proposition~\ref{Prop:MaximumPrincipleForOddFunctions}. In Section~\ref{Sec:EnergyForOddF} we study the energy functional associated to \eqref{Eq:NonlocalAllenCahn} and in Section~\ref{Sec:EnergyEstimate} we establish the energy estimate stated in Theorem~\ref{Th:EnergyEstimate}. Finally, in Section~\ref{Sec:Existence} we prove the existence of a saddle-shaped solution to the integro-differential Allen-Cahn equation. At the end of the paper there are three appendices. Appendix~\ref{Sec:AuxiliaryResults} is devoted to some results on convex functions, and Appendix~\ref{Sec:AuxiliaryResults2} contains some auxiliary computations. Both are used in the proof of Theorem~\ref{Th:SufficientNecessaryConditions}. In Appendix~\ref{Sec:stcomputations} we include some results and expressions in $(s,t)$ variables for future reference.

\section{Rotation invariant operators acting on doubly radial odd functions}
\label{Sec:OperatorOddF}

This section is devoted to study rotation invariant operators of the class $\lcal_0$ when they act on doubly radial odd functions. First, we deduce an alternative expression for the operator in terms of a doubly radial kernel $\overline{K}$. Then, we present necessary and sufficient conditions on the kernel $K$ in order to \eqref{Eq:KernelInequality} hold (we establish Theorem~\ref{Th:SufficientNecessaryConditions}). Finally, we show two maximum principles for doubly radial odd functions (Proposition~\ref{Prop:MaximumPrincipleForOddFunctions}).

\subsection{Alternative expressions for the operator $L_K$}

The main purpose of this subsection is to deduce an alternative expression for a rotation invariant operator $L_K \in \lcal_0$ acting on doubly radial functions. This expression is more suitable to work with and it will be used throughout the paper. Our first remark is that if $w$ is invariant by $O(m)^2$, the same holds for $L_Kw$. Indeed, for every $R \in O(m)^2$,
\begin{align*}
L_K w (Rx)
& = \int_{\R^{2m}} \{w(Rx) - w(y)\} K(|Rx - y|)  \d y\\
& = \int_{\R^{2m}} \{w(Rx) - w(R\tilde{y})\} K(|Rx - R\tilde{y}|) \d \tilde{y}\\
& = \int_{\R^{2m}} \{w(x) - w(\tilde{y})\} K(|x-\tilde{y}|) \d \tilde{y}\\
& = L_K w (x)\,.
\end{align*}
Here we have used the change $y = R\tilde{y}$ and the fact that $w(R \cdot) = w(\cdot)$ for every $R\in O(m)^2$.

Next, we present an alternative expression for the operator $L_K $ acting on doubly radial functions. This expression involves the new kernel $\overline{K}$, which is also doubly radial.

\begin{lemma} \label{Lemma:AlternativeOperatorExpression}
	Let $L_K \in \lcal_0(2m,\s)$ have a radially symmetric kernel $K$, and let $w$ be a doubly radial function such that $L_K w$ is well-defined. Then, $L_K w$ can be expressed as
	$$
	L_K w(x) = \int_{\R^{2m}} \{w(x) - w(y)\} \overline{K}(x,y) \d y
	$$
	where $\overline{K}$ is symmetric, invariant by $O(m)^2$ in both arguments, and it is defined by
	\begin{equation*}
	\overline{K}(x,y) := \average_{O(m)^2} K(|Rx - y|)\d R\,.
	\end{equation*}
	Here, $\d R$ denotes integration with respect to the Haar measure on $O(m)^2$.
\end{lemma}

Recall (see for instance \cite{Nachbin}) that the Haar measure on $O(m)^2$ exists and it is unique up to a
multiplicative constant. Let us state next the properties of this measure that will be used in the rest of the
paper. In the following, the Haar measure is denoted by $\mu$. First, since $O(m)^2$ is a compact
group, it is unimodular (see Chapter~II, Proposition~ 13 of \cite{Nachbin}). As a consequence, the
measure $\mu$ is left and right invariant, that is, $\mu(R\Sigma) = \mu(\Sigma) = \mu(\Sigma R) $
for every subset $\Sigma \subset O(m)^2$ and every $R\in O(m)^2$. Moreover, it holds
\begin{equation}
\label{Eq:Unimodular}
\average_{O(m)^2} g(R^{-1}) \d R = \average_{O(m)^2} g(R) \d R
\end{equation}	
for every $g\in L^1(O(m)^2)$ ---see \cite{Nachbin} for the details.

\begin{proof}[Proof of Lemma~\ref{Lemma:AlternativeOperatorExpression}]
	Since $L_K w (x) = L_K w (Rx)$ for every $R\in O(m)^2$, by taking the mean over all the transformations in $O(m)^2$, we get
	\begin{align*}
	L_K w(x) &= \average_{O(m)^2} L_K w(Rx)\d R =  \average_{O(m)^2} \int_{\R^{2m}} \{w(x) - w(y)\}K(|Rx - y|) \d y \d R\\
	&= \int_{\R^{2m}} \{w(x) - w(y)\}  \average_{O(m)^2} K(|Rx - y|) \d R  \d y \\
	&= \int_{\R^{2m}} \{w(x) - w(y)\}  \overline{K}(x,y) \d y\,.
	\end{align*}
	Now, we show that $\overline{K}$ is symmetric. Using property \eqref{Eq:Unimodular}, we get
	\begin{align*}
	\overline{K}(y,x) &= \average_{O(m)^2} K(|R y - x|)\d R = \average_{O(m)^2} K(|R^{-1} (R y - x)|)\d R \\
	&= \average_{O(m)^2} K(|R^{-1}x-y)|)\d R = \overline{K}(x,y)\,.
	\end{align*}
	It remains to show that
	$\overline{K}$ is invariant by $O(m)^2$ in its two arguments. By the symmetry, it is enough to
	check it for the first one. Let $\tilde{R} \in O(m)^2$. Then,
	$$
	\overline{K} (\tilde{R}x, y) = \average_{O(m)^2} K(|R \tilde{R} x - y|)\d R  = \average_{O(m)^2} K(|R x - y|)\d R = \overline{K} (x, y)\,,
	$$
	where we have used the right invariance of the Haar measure.
\end{proof}

In the following lemma we present some properties of the involution $(\cdot)^\star$ defined by \eqref{Eq:DefStar} and its relation with the doubly radial kernel $\overline{K}$ and the transformations of $O(m)^2$. In particular, in the proof of Theorem~\ref{Th:SufficientNecessaryConditions} it will be useful to consider the following transformation. For every $R\in O(m)^2$, we define  $R_\star\in O(m)^2$ by 
\begin{equation}
\label{Eq:DefRStar}
R_\star := (R(\cdot)^\star)^\star\,.
\end{equation}
Equivalently, if $R = (R_1, R_2)$ with $R_1$, $R_2 \in O(m)$, then $R_\star = (R_2, R_1)$.

\begin{lemma}
	\label{Lemma:PropertiesStar}
	Let $(\cdot)^\star: \R^{2m} \to \R^{2m}$ be the involution defined by $x^\star = (x',x'')^\star = (x'', x')$
	---see \eqref{Eq:DefStar}.
	Then,
	\begin{enumerate}
		\item
		The Haar integral on $O(m)^2$ has the following invariance:
		\begin{equation}
		\label{Eq:InvarianceByStar}
		\int_{O(m)^2} g(R_\star) \d R = \int_{O(m)^2} g(R) \d R \,,
		\end{equation}
		for every $g \in L^1(O(m)^2)$.
		\item $\overline{K}(x^\star,y) = \overline{K} (x,y^\star)$. As a consequence, $\overline{K}(x^\star,y^\star) = \overline{K} (x,y)$.
	\end{enumerate}
\end{lemma}

\begin{proof}
	The first statement is easy to check by using Fubini:
	\begin{align*}
	\int_{O(m)^2} g(R_\star) \d R & = \int_{O(m)} \!\! \d R_1 \int_{O(m)} \!\! \d R_2 \ \ g(R_2, R_1)  =  \int_{O(m)} \!\! \d R_2 \int_{O(m)} \!\! \d R_1 \ \ g(R_2, R_1) \\
	& =  \int_{O(m)} \!\! \d R_1 \int_{O(m)} \!\! \d R_2 \ \ g(R_1, R_2)  =  \int_{O(m)^2} g(R) \d R\,.
	\end{align*}
	
	To show the second statement, we use the definition of $R_\star$ and \eqref{Eq:InvarianceByStar}
	to see that
	\begin{align*}
	\overline{K}(x^\star,y) &= \average_{O(m)^2} K(|Rx^\star - y|) \d R = \average_{O(m)^2} K(|(Rx^\star - y)^\star|) \d R \\
	&= \average_{O(m)^2} K(|(R x^\star)^\star - y^\star|) \d R = \average_{O(m)^2} K(|R_\star x - y^\star|) \d R \\
	&= \average_{O(m)^2} K(|Rx - y^\star|) \d R = \overline{K}(x,y^\star)\,.
	\end{align*}
	As a consequence, we have that $\overline{K}(x^\star,y^\star) = \overline{K}(x,(y^\star)^\star) = \overline{K}(x,y)$.
\end{proof}

To conclude this subsection, we present two alternative expressions for the operator $L_K$ when it acts on doubly radial odd functions. These expressions are suitable in the rest of the paper and also in the forthcoming one \cite{FelipeSanz-Perela:IntegroDifferentialII}, since the integrals appearing in the expression are computed only in $\ocal$, and this is important to prove maximum principle and other properties.

\begin{lemma}
	\label{Lemma:OperatorOddF}
	Let $w$ be a doubly radial function which is odd with respect to the Simons cone. Let $L_K \in \lcal_0(2m,\s,\lambda, \Lambda)$ be a rotation invariant operator and let $L_K^\ocal$ be defined by \eqref{Eq:OperatorOddF}. 
	
	Then, 
	\begin{align*}
	L_K w (x) &= \int_{\ocal} \{w(x) - w(y) \} \overline{K}(x, y) \d y +  \int_{\ocal} \{w(x) + w(y) \} \overline{K}(x, y^\star) \d y \\
	&= \int_{\ocal} \{w(x) - w(y) \} \{\overline{K}(x, y) - \overline{K}(x, y^\star)  \} \d y +  2 w(x) \int_{\ocal} \overline{K}(x, y^\star) \d y \,.
	\end{align*}
	In particular, the second equality shows that $L_K w (x) = L_K^\ocal w(x)$. 	Moreover,
	\begin{equation}
	\label{Eq:ZeroOrderTerm}
	\dfrac{1}{C} \dist(x,\ccal)^{-2\s} \leq \int_{\ocal} \overline{K}(x, y^\star) \d y \leq C \dist(x,\ccal)^{-2\s},
	\end{equation}
	where $C>0$ is a constant depending only on $m, \s, \lambda$, and $\Lambda$.
\end{lemma}

\begin{proof}
	The first statement is just a computation. Indeed,  using the change of variables  $\bar{y} = y^\star$ and the odd symmetry of $w$, we see that
	\begin{align*}
	\int_{\ical}  \{w(x) - w(y) \} \overline{K}(x, y)\d y &= \int_{\ocal} \{w(x) - w(y^\star) \}\overline{K}(x, y^\star)\d y \\
	&= \int_{\ocal} \{w(x) + w(y) \}\overline{K}(x, y^\star)\d y\,.
	\end{align*}
	Hence,
	\begin{align*}
	L_K w (x) &= \int_{\R^{2m}}  \{w(x) - w(y) \} \overline{K}(x, y)\d y \\
	&= \int_{\ocal}  \{w(x) - w(y) \} \overline{K}(x, y)\d y +\int_{\ical}  \{w(x) - w(y) \} \overline{K}(x, y)\d y \\
	&= \int_{\ocal} \{w(x) - w(y) \} \overline{K}(x, y) \d y +  \int_{\ocal} \{w(x) + w(y) \} \overline{K}(x, y^\star) \d y \,.
	\end{align*}
	By adding and subtracting $w(x)\overline{K}(x, y^\star)$ in the last integrand, we immediately deduce
	$$
	L_K w (x) =  \int_{\ocal} \{w(x) - w(y) \} \{\overline{K}(x, y) - \overline{K}(x, y^\star)  \} \d y +  2 w(x) \int_{\ocal} \overline{K}(x, y^\star) \d y\,.
	$$
	Note that we can add and subtract the term $w(x)\overline{K}(x, y^\star)$  since it is integrable with respect to $y$ in $\ocal$. This is a consequence of \eqref{Eq:ZeroOrderTerm}.
	
	Let us show now \eqref{Eq:ZeroOrderTerm}. In the following arguments we will use the letters $C$ and $c$ to denote positive constants, depending only on $m, \s, \lambda$, and $\Lambda$, that may change their values in each inequality. 
	
	On the one hand, for the upper bound in \eqref{Eq:ZeroOrderTerm} we only need to use the ellipticity of the kernel and the inclusion $\ical \subset \{y\in\R^{2m}:|x-y|\geq \dist(x,\ccal)\}$ for every $x\in \ocal$. Indeed,
	\begin{align*}
	\int_{\ocal} \overline{K}(x, y^\star) \d y &=  \int_{\ocal} K(|x-y^\star|) \d y = \int_{\ical} K(|x-y|) \d y \leq \int_{|x-y|\geq \dist(x,\ccal)} K(|x-y|) \d y \\
	&\leq C \int_{|x-y|\geq \dist(x,\ccal)} |x-y|^{-2m-2\s} \d y = C \int_{\dist(x,\ccal)}^\infty \rho^{-1-2s} \d \rho \\
	&= C \dist(x,\ccal)^{-2s}\,.
	\end{align*}
	
	On the other hand, for the lower bound in \eqref{Eq:ZeroOrderTerm}, let $\overline{x}\in \ccal$ be such that $|x-\overline{x}|=\dist(x,\ccal)$. Then, given $y\in B_{\dist(x,\ccal)}(\overline{x})$, it is clear that $|x-y|\leq |x-\overline{x}|+|\overline{x}-y|\leq 2\dist(x,\ccal)$. Therefore, we have
	\begin{align*}
	\int_{\ocal} \overline{K}(x, y^\star) \d y &=  \int_{\ical} K(|x-y|) \d y \geq c \int_{\ical} |x-y|^{-2m-2\s} \d y \\
	&\geq c \int_{B_{\dist(x,\ccal)}(\overline{x})\cap \ical} |x-y|^{-2m-2\s} \d y\\
	&\geq c \, (2\dist(x,\ccal))^{-2m-2\s} |B_{\dist(x,\ccal)}(\overline{x})\cap \ical| = c \,\dist(x,\ccal)^{-2\s}.
	\end{align*}
	Here we have used a property of the Simons cone: $|B_R(z)\cap \ical|=1/2|B_R|$ for every $z\in \ccal$ (see Lemma 2.5 in \cite{Felipe-Sanz-Perela:SaddleFractional} for the proof).
\end{proof}

\subsection{Necessary and sufficient conditions for ellipticity}

In this subsection, we establish Theorem~\ref{Th:SufficientNecessaryConditions}. As we have mentioned in the introduction, the kernel inequality \eqref{Eq:KernelInequality} is crucial in the rest of the results of this paper, as well as in the ones in \cite{FelipeSanz-Perela:IntegroDifferentialII}. We will see in the next subsection that this inequality guarantees that the operator $L_K$ has a maximum principle for odd functions (see Proposition~\ref{Prop:MaximumPrincipleForOddFunctions}).

First, we give a sufficient condition on a radially symmetric kernel $K$ so that $\overline{K}$ satisfies \eqref{Eq:KernelInequality}. It is the following result.

\begin{proposition}
	\label{Prop:KernelInequalitySufficientCondition} 
	Let $K:(0,+\infty) \to \R$ define a positive radially symmetric kernel $K(|x-y|)$ in $\R^{2m}$. Define $\overline{K} : \R^{2m}\times \R^{2m} \to \R$ by \eqref{Eq:KbarDef'}. Assume that $K(\sqrt{\cdot})$ is strictly convex in $(0,+\infty)$. Then, the associated kernel $\overline{K}$ satisfies
	\begin{equation}
	\label{Eq:KernelInequalityBis}
	\overline{K}(x,y) > \overline{K}(x, y^\star) \quad \text{ for every }x,y \in \ocal\,.
	\end{equation}
\end{proposition}

\begin{proof}
	Since $\overline{K}$ is invariant by $O(m)^2$, it is enough to choose a unitary vector $e \in \Sph^{m-1}$ and show \eqref{Eq:KernelInequalityBis} for points $x, y\in \ocal$ of the form $x = (|x'|e, |x''|e)$ and $y = (|y'|e, |y''|e)$.
	
	Now, define
	\begin{equation}
	\label{Eq:DefQ}
	\begin{split}
	Q_1 &:= \setcond{R = (R_1,R_2) \in O(m)^2}{e\cdot R_1 e > |e\cdot R_2 e|},\\
	Q_2 &:= \setcond{R = (R_1,R_2) \in O(m)^2}{e\cdot R_2 e > |e\cdot R_1 e|} = (Q_1)_\star,\\
	Q_3 &:= \setcond{R = (R_1,R_2) \in O(m)^2}{e\cdot R_1 e < -|e\cdot R_2 e|} = -Q_1,\\
	Q_4 &:= \setcond{R = (R_1,R_2) \in O(m)^2}{e\cdot R_2 e < - |e\cdot R_1 e|} = -(Q_1)_\star.
	\end{split}
	\end{equation}
	Recall that given $R=(R_1,R_2)\in O(m)^2$, then $R_\star=(R_2,R_1)\in O(m)^2$ ---see \eqref{Eq:DefRStar}. Moreover, note that the sets $Q_i$ are disjoint, have the same measure and cover all $O(m)^2$ up to a set of measure zero. 
	
	Therefore,
	\begin{align*}
	4\overline{K} (x, y) &= 4\average_{O(m)^2} K(|x - R y|)\d R \\
	& = \average_{Q_1} K(|x - R y|)\d R + \average_{Q_2} K(|x - R y|)\d R \\
	& \quad \quad
	+ \average_{Q_3} K(|x - R y|)\d R +
	\average_{Q_4} K(|x - R y|)\d R \\
	&= \average_{Q_1} \{K(|x - R y|) + K(|x + R y|) \\
	&\quad \quad + K(|x - R_\star y|) + K(|x + R_\star y|)\}\d R
	\end{align*}
	and
	\begin{align*}
	4\overline{K} (x, y^\star) &= 4\average_{O(m)^2} K(|x - R y^\star|)\d R \\
	& = \average_{Q_1} \{K(|x - R y^\star|) + K(|x + R y^\star|) \\
	&\quad \quad + K(|x - R_\star y^\star|) + K(|x + R_\star y^\star|)\}\d R.
	\end{align*}
	Thus, if we prove
	\begin{equation}
	\label{Eq:InequalityIntegrandKernelInequalityProof}
	\begin{split}
	K(|x - R y|) + K(|x + R y|) + K(|x - R_\star y|) + K(|x + R_\star y|)
	\quad \quad \quad \quad 
	\\
	\geq
	K(|x - R y^\star|) + K(|x + R y^\star|)+K(|x - R_\star y^\star|) + K(|x + R_\star y^\star|)\,,
	\end{split}
	\end{equation}
	for every $R\in Q_1$, we immediately deduce \eqref{Eq:KernelInequalityBis} with a non strict inequality. To see that it is indeed a strict one, we will show that the inequality in \eqref{Eq:InequalityIntegrandKernelInequalityProof} is strict for every $R \in Q_1$.

	For a short notation, we call
	\begin{equation}
	\label{Eq:DefAlphaBeta}
	\alpha := e \cdot R_1 e  \quad \text{ and } \quad \beta := e \cdot R_2 e\,.
	\end{equation}
	Now, note that since  $x = (|x'|e, |x''|e)$ and $y = (|y'|e, |y''|e)$, we have
	\begin{align*}
	|x \pm Ry|^2&= |x' \pm R_1y'|^2 + |x'' \pm R_2y''|^2 \\
	&= |x'|^2 + |y'|^2 \pm 2 x'\cdot R_1 y' +  |x''|^2 + |y''|^2 \pm 2 x''\cdot R_2 y''\\
	&= |x|^2 + |y|^2 \pm 2 |x'||y'| \alpha \pm 2 |x''||y''| \beta.
	\end{align*}
	Similarly,
	$$
	|x \pm R_\star y|^2 =  |x|^2 + |y|^2 \pm 2 |x'||y'| \beta \pm 2 |x''||y''|\alpha,
	$$
	$$
	|x \pm R y^\star|^2 =  |x|^2 + |y|^2 \pm 2 |x'||y''| \alpha \pm 2 |x''||y'|\beta,
	$$
	and
	$$
	|x \pm R_\star y^\star|^2 = |x|^2 + |y|^2 \pm 2 |x'||y''| \beta \pm 2 |x''||y'| \alpha.
	$$
	
	We define now
	$$
	g(\tau) := K \bpar{\sqrt{|x|^2 + |y|^2 + 2 \tau }} + K \bpar{\sqrt{|x|^2 + |y|^2 - 2 \tau}}.
	$$
	Thus, proving \eqref{Eq:InequalityIntegrandKernelInequalityProof} is equivalent to show that, for every $\alpha$, $\beta \in [-1,1]$ such that $\alpha > |\beta|$, it holds
	\begin{equation}
	\label{Eq:InequalityIntegrandKernelInequalityProof2}
	\begin{split}
	g\Big(|x'||y'| \alpha + |x''||y''| \beta \Big)
	+ g\Big(|x'||y'| \beta + |x''||y''| \alpha \Big) \hspace{2cm}
	\\ \geq
	g\Big(|x'||y''| \alpha + |x''||y'|\beta \Big)
	+ g\Big(|x'||y''| \beta + |x''||y'| \alpha \Big)\,.
	\end{split}
	\end{equation}
	
	Let
	$$
	\begin{array}{cc}
	A_{\alpha,\beta} := |x'||y'|  \alpha + |x''||y''|\beta \,, &
	B_{\alpha,\beta} := |x'||y''| \alpha + |x''||y'| \beta \,, \\
	C_{\alpha,\beta} := |x''||y'| \alpha + |x'||y''| \beta \,, &
	D_{\alpha,\beta} := |x''||y''|\alpha + |x'||y'|  \beta \,.
	\end{array}
	$$
	With this notation and taking into account that $g$ is even,
	\eqref{Eq:InequalityIntegrandKernelInequalityProof2} is equivalent to
	\begin{equation}
	\label{Eq:InequalityIntegrandKernelInequalityProof3}
	g(|A_{\alpha,\beta}|) + g(|D_{\alpha,\beta}|) \geq g(|C_{\alpha,\beta}|) + g(|B_{\alpha,\beta}|)\,,
	\end{equation}
	for every $\alpha$, $\beta \in [-1,1]$ such that $\alpha > |\beta|$. Note that $g$ is defined in the open interval $I = (-(|x|^2 + |y|^2)/2,\ (|x|^2 + |y|^2)/2)$ and that $A_{\alpha,\beta}$, $B_{\alpha,\beta}$, $C_{\alpha,\beta}$, $D_{\alpha,\beta} \in I$.
	
	To show \eqref{Eq:InequalityIntegrandKernelInequalityProof3}, we use Proposition~\ref{Prop:EquivalenceK(sqrt)Convex<->Inequality} of the Appendix~\ref{Sec:AuxiliaryResults}. There, it is stated that in order to establish \eqref{Eq:InequalityIntegrandKernelInequalityProof3} it is enough to check that
	$$
	\begin{cases}
	|A_{\alpha,\beta}| \geq |B_{\alpha,\beta}|,\ \ |A_{\alpha,\beta}| \geq |C_{\alpha,\beta}|, \ \ |A_{\alpha,\beta}| \geq |D_{\alpha,\beta}|\,, \\
	|A_{\alpha,\beta}| + |D_{\alpha,\beta}| \geq |B_{\alpha,\beta}| + |C_{\alpha,\beta}|\,.
	\end{cases}
	$$
	The verification of these inequalities is a simple but tedious computation and it is presented in Appendix~\ref{Sec:AuxiliaryResults2} ---see point (1) of Lemma~\ref{Lemma:ComputationABCD}. Once this is proved, we deduce \eqref{Eq:InequalityIntegrandKernelInequalityProof3} by Proposition~\ref{Prop:EquivalenceK(sqrt)Convex<->Inequality}.
	
	To finish, we see that the inequality in \eqref{Eq:InequalityIntegrandKernelInequalityProof3} is always strict for every $\alpha$, $\beta \in [-1,1]$ such that $\alpha > |\beta|$ (that corresponds to $Q_1$). By contradiction, assume that equality holds in \eqref{Eq:InequalityIntegrandKernelInequalityProof3}. Thus, by Proposition~\ref{Prop:EquivalenceK(sqrt)Convex<->Inequality}, if follows that the sets $\{|A_{\alpha,\beta}|, |D_{\alpha,\beta}|\} $ and $\{|B_{\alpha,\beta}|,|C_{\alpha,\beta}|\}$ coincide. This fact and point (2) of Lemma~\ref{Lemma:ComputationABCD} yield $\alpha = \beta = 0$, a contradiction. Thus, the inequality in \eqref{Eq:InequalityIntegrandKernelInequalityProof3} is strict, as well as the inequality in \eqref{Eq:InequalityIntegrandKernelInequalityProof}. This leads to \eqref{Eq:KernelInequalityBis}.
\end{proof}

Now, we give a necessary condition on the kernel $K$ so that inequality \eqref{Eq:KernelInequality} holds.

\begin{proposition}
	\label{Prop:KernelInequalityNecessaryCondition} Let $K:(0,+\infty) \to \R$ define a positive radially symmetric kernel $K(|x-y|)$ in $\R^{2m}$. Define $\overline{K} : \R^{2m}\times \R^{2m} \to \R$ by \eqref{Eq:KbarDef'}. 
	
	If
	\begin{equation}
	\label{Eq:KernelInequalityAE}
	\overline{K}(x,y) > \overline{K}(x, y^\star) \quad \text{ for almost every }x,y \in \mathcal{O}\,,
	\end{equation}
	then $K(\sqrt{\cdot})$ cannot be concave in any open interval $I\subset [0,+\infty)$.
\end{proposition}

\begin{proof}
	It suffices to show that if there exists an open interval where $K(\sqrt{\cdot})$ is concave, then we can find a nonempty open set in $\ocal \times \ocal$ where \eqref{Eq:KernelInequalityAE} is not satisfied.
	
	Let $\ell_2>\ell_1>0$ be such that $K(\sqrt{\cdot})$ is concave in $(\ell_1,\ell_2)$ and define the set $\Omega_{\ell_1,\ell_2}\subset \R^{4m}$ as the points $(x,y)\in \ocal\times \ocal$ satisfying
	\begin{equation}
	\label{Eq:OmegaSetDefinition}
	\beqc{\PDEsystem}
	(|x'|-|y'|)^2+(|x''|-|y''|)^2&>&\ell_1,\\
	(|x'|+|y'|)^2+(|x''|+|y''|)^2&<&\ell_2.
	\eeqc
	\end{equation}
	
	First, it is easy to see that $\Omega_{\ell_1,\ell_2}$ is a nonempty open set. In fact, points of the form $(x',0,y',0)\in (\R^m)^4$ such that $(|x'|-|y'|)^2>\ell_1$ and $(|x'|+|y'|)^2 <\ell_2$ belong to $\Omega_{\ell_1,\ell_2}$. 
	
	We need to prove that $\overline{K}(x,y) \leq \overline{K}(x, y^\star)$ in $\Omega_{\ell_1,\ell_2}$ for any $(x,y)\in \ocal\times \ocal$ satisfying \eqref{Eq:OmegaSetDefinition}. For such points, we are going to show, as in the previous proof, that
	\begin{equation}
	\label{Eq:InequalityIntegrandKernelInequalityProof4}
	\begin{split}
	K(|x - R y|) + K(|x + R y|) + K(|x - R_\star y|) + K(|x + R_\star y|)
	\quad \quad \quad \quad \quad \quad
	\\
	\leq
	K(|x - R y^\star|) + K(|x + R y^\star|)+K(|x - R_\star y^\star|) + K(|x + R_\star y^\star|)\,, 
	\end{split}
	\end{equation}
	for any $R\in Q_1$, where $Q_1$ is defined in \eqref{Eq:DefQ} (see the proof of Proposition~\ref{Prop:KernelInequalitySufficientCondition}). As before, we can assume that $x$ and $y$ are of the form $x = (|x'|e, |x''|e)$ and $y = (|y'|e, |y''|e)$, with $e \in \Sph^{m-1}$ an arbitrary unitary vector. Then, by defining $\alpha$ and $\beta$ as in \eqref{Eq:DefAlphaBeta}, we see that proving \eqref{Eq:InequalityIntegrandKernelInequalityProof4} is equivalent to establish that
	\begin{equation}
	\label{Eq:InequalityIntegrandKernelInequalityProof5}
	g(A_{\alpha,\beta}) + g(D_{\alpha,\beta}) \leq g(B_{\alpha,\beta}) + g(C_{\alpha,\beta})\,,
	\end{equation}
	for every $\alpha, \beta \in [-1,1]$ such that $\alpha>|\beta|$, where
	$$
	\begin{array}{cc}
	A_{\alpha,\beta} = |x'||y'|  \alpha + |x''||y''|\beta \,, &
	B_{\alpha,\beta} = |x'||y''| \alpha + |x''||y'| \beta \,, \\
	C_{\alpha,\beta} = |x''||y'| \alpha + |x'||y''| \beta \,, &
	D_{\alpha,\beta} = |x''||y''|\alpha + |x'||y'|  \beta \,.
	\end{array}
	$$
	and
	\begin{align*}
	g(\tau) &= K\left( \sqrt{|x|^2+|y|^2+2\tau} \right) + K\left( \sqrt{|x|^2+|y|^2-2\tau} \right).
	\end{align*}

	Now, by \eqref{Eq:OmegaSetDefinition}, we have $\ell_1 < |x|^2+|y|^2 <\ell_2$. As a consequence of this and the concavity of $K(\sqrt{\cdot})$ in $(\ell_1,\ell_2)$, it is easy to see (by using Lemma~\ref{Lemma:ConvexFunctions} stated for $-h$, a concave function, instead of $h$) that $g$ is concave in $ \left( -\overline{\ell}, \overline{\ell}\right) $, and decreasing in $(0,\overline{\ell})$, where 
	$$
	\overline{\ell} := \min{\left\{\frac{\ell_2-|x|^2-|y|^2}{2},\frac{|x|^2+|y|^2-\ell_1}{2}\right\}}.$$
	Note that, since $\ell_1 < |x|^2+|y|^2 <\ell_2$, we have $\overline{\ell}>0$.

	We claim that $A_{\alpha,\beta}, B_{\alpha,\beta}, C_{\alpha,\beta}$, and $D_{\alpha,\beta}$ belong to $(-\overline{\ell},\overline{\ell})$ for every $\alpha, \beta \in [-1,1]$ such that $\alpha>|\beta|$. Indeed, it is easy to check that for every $\alpha, \beta \in [-1,1]$ such that $\alpha>|\beta|$, the numbers $A_{\alpha,\beta}, B_{\alpha,\beta}, C_{\alpha,\beta}$, and $D_{\alpha,\beta}$ belong to the open interval $(-|x'||y'|-|x''||y''|,|x'||y'|+|x''||y''|)$. Furthermore, since $x,y \in \Omega_{\ell_1,\ell_2}$, we obtain from \eqref{Eq:OmegaSetDefinition} that
	$$
	\beqc{\PDEsystem}
	|x'||y'|+|x''||y''|&<&\dfrac{\ell_2-|x|^2-|y|^2}{2}\\
	|x'||y'|+|x''||y''|&<&\dfrac{|x|^2+|y|^2-\ell_1}{2}
	\eeqc 
	$$
	and thus $ |x'||y'|+|x''||y''|<\overline{\ell}$ and the claim is proved.
	
	Finally, by applying Lemma~\ref{Lemma:ConvexFunctions} to the function $-g$ in $(0,\overline{\ell})$ (using again  point (1) of Lemma~\ref{Lemma:ComputationABCD}), we obtain that inequality \eqref{Eq:InequalityIntegrandKernelInequalityProof5} is satisfied, which yields \eqref{Eq:InequalityIntegrandKernelInequalityProof4}. Finally, by integrating \eqref{Eq:InequalityIntegrandKernelInequalityProof4} with respect to all the rotations $R\in Q_1$ we get $$ \overline{K}(x,y) \leq \overline{K}(x, y^\star),$$ for every $(x,y)\in \Omega_{\ell_1,\ell_2}$, contradicting \eqref{Eq:KernelInequalityAE}.
\end{proof}

From the two previous results, Theorem~\ref{Th:SufficientNecessaryConditions} follows immediately.

\begin{proof}[Proof of Theorem~\ref{Th:SufficientNecessaryConditions}]
	The first statement is exactly the same as Proposition~\ref{Prop:KernelInequalitySufficientCondition}. Assume now that $K$ is a $C^2$ function and that \eqref{Eq:KernelInequality} holds. Then, by Proposition~\ref{Prop:KernelInequalityNecessaryCondition}, $h(\cdot) := K(\sqrt{\cdot})$ is not concave in any interval of $[0,+\infty)$. Therefore, we cannot have $h''<0$ at any point. Thus, $h''\geq 0$ in $[0,+\infty)$ or, in other words, $h'$ is nondecreasing. Using again that $h$ is not concave in any interval, we deduce that $h'$ must be, in fact, increasing. It follows that $h(\cdot) = K(\sqrt{\cdot})$ is strictly convex as defined after the statement of Theorem~\ref{Th:SufficientNecessaryConditions}.
\end{proof}

\begin{remark}
	Note that a priori we cannot relax the $K\in C^2$ assumption in the necessary condition of Theorem~\ref{Th:SufficientNecessaryConditions}, since there are $C^1$ functions that are neither convex nor concave in any interval (they can be constructed as a primitive of a Weierstrass function, whose graph is a non rectifiable curve with fractal dimension). Besides these ``exotic'' examples, there are also simple radially symmetric kernels $K$ that are not $C^1$ for which we do not know if the positivity condition \eqref{Eq:KernelInequality} holds. For instance, given $0<\s<1$, if we consider the kernel
	$$ K(\tau) = \frac{1}{\tau^{2m+2\s}} \chi_{(0,1)}(\tau)+\frac{1}{10\tau^{2m+2\s}-9} \chi_{[1,+\infty)}(\tau), $$
	it is easy to check that $K$ is continuous and decreasing but $K(\sqrt{\tau})$ is not convex in $(0,+\infty)$ even though it does not have any interval of concavity (see Figure~\ref{Fig:Grafica}).
	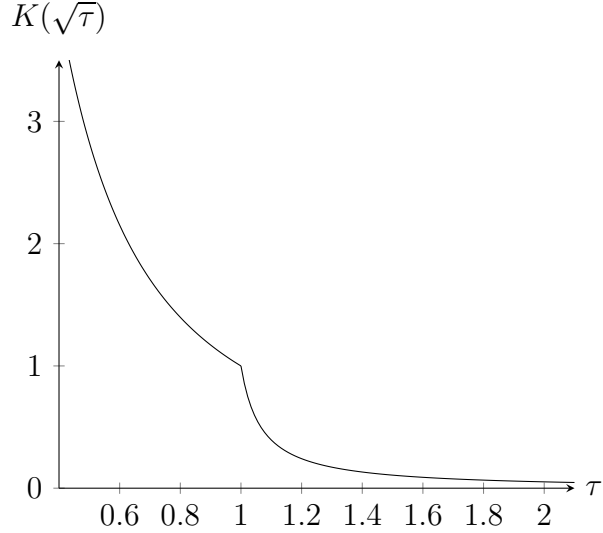
\begin{figure}
		\centering
		\begin{tikzpicture}
		\begin{axis}[
		axis x line=middle, axis y line=left,
		every axis x label/.style={at={(current axis.right of origin)},anchor=west},
		every axis y label/.style={at={(current axis.north west)},above=2mm},
		ymin=0, ymax=3.5, ylabel=$K(\sqrt{\tau})$,
		xmin=0.4, xmax=2.1, xlabel=$\tau$
		]
		\addplot[domain=0.4:1, samples=100] {1/(x^1.5)};
		\addplot[domain=1:2.1, samples=100] {0.1/(x^1.5-0.9)};
		\end{axis}
		\end{tikzpicture}
		\caption{An example of kernel $K(\sqrt{\tau})$ ($m=1$ and $\s=1/2$) which is not strictly convex in $(0,+\infty)$ but does not have any interval of concavity. }
		\label{Fig:Grafica}
	\end{figure}
\end{remark}

\subsection{Maximum principles for doubly radial odd functions}

In this subsection we prove Proposition~\ref{Prop:MaximumPrincipleForOddFunctions}, a weak and a strong maximum principles for doubly radial functions that are odd with respect to the Simons cone. The formulation of these maximum principles is very suitable since all the hypotheses refer to the set $\ocal$ and not $\R^{2m}$. The key ingredient in the proofs is the kernel inequality \eqref{Eq:KernelInequality}.

\begin{proof}[Proof of Proposition~\ref{Prop:MaximumPrincipleForOddFunctions}]
	$(i)$ By contradiction, suppose that $u$ takes negative values in $\Omega$. Under the hypotheses we are assuming, a negative minimum must be achieved. Thus, there exists $x_0\in \Omega$ such that
	$$
	u(x_0) = \min_{\Omega} u =: m < 0\,.
	$$
	Then, using the expression of $L_K$ for odd functions (see Lemma~\ref{Lemma:OperatorOddF}), we have
	$$
	L_K u (x_0) = \int_{\ocal} \{m - u(y) \} \{\overline{K}(x_0, y) - \overline{K}(x_0, y^\star)  \} \d y +  2 m \int_{\ocal} \overline{K}(x_0, y^\star) \d y\,.
	$$
	Now, since $m - u(y) \leq 0$ in $\ocal$, $m<0$, $c\geq 0$, and $\overline{K}(x_0, y) \geq \overline{K}(x_0, y^\star)>0$ ---by \eqref{Eq:KernelInequality}---, we get
	$$
	0 \leq L_K  u(x_0) + c(x_0) u(x_0) \leq m \left(2\int_{\ocal} \overline{K}(x_0, y^\star) \d y + c(x_0)\right)  < 0\,,
	$$
	a contradiction.
	
	$(ii)$ 
	Assume that $u \not \equiv 0$ in $\ocal$. We shall prove that $u > 0$ in $\Omega$. By contradiction, assume that there exists a point $x_0\in \Omega$ such that $u(x_0)= 0$. Then, using the expression of $L_K $ for odd functions given in Lemma~\ref{Lemma:OperatorOddF}, the kernel inequality \eqref{Eq:KernelInequality}, and the fact that $u\geq 0$ in $\ocal$, we obtain
	$$
	0 \leq L_K u(x_0) + c(x_0) u(x_0) = - \int_{\ocal} u(y)\big \{\overline{K}(x_0, y) - \overline{K}(x_0, y^\star) \big \}\d y < 0\,,
	$$
	a contradiction.
\end{proof}

\begin{remark}
	Note that since the operator $L_K$ includes a term of order zero with positive coefficient in addition to the integro-differential part, the condition $c\geq 0$ in point $(i)$ of the previous proposition can be slightly relaxed. Indeed, following the proof of the result, we see that
	$$ c(x) > -2\int_{\ocal} \overline{K}(x, y^\star) \d y $$
	suffices.
	This hypothesis seems hard to be checked for applications apart from the case $c\geq 0$. Nevertheless, recall that by Lemma~\ref{Lemma:OperatorOddF} we have an explicit lower bound for the quantity $ \int_{\ocal} \overline{K}(x, y^\star) \d y $ in terms of the function $\dist(x,\ccal)$. This fact will be crucial for establishing a maximum principle in ``narrow'' sets close to the Simons cone ---see Proposition~6.2 in part II \cite{FelipeSanz-Perela:IntegroDifferentialII}.
\end{remark}

\section{The energy functional for doubly radial odd functions}
\label{Sec:EnergyForOddF}

This section is devoted to the energy functional associated to the semilinear equation \eqref{Eq:NonlocalAllenCahn}. We first define appropriately the functional spaces where we are going to apply classic techniques of calculus of variations. Next we rewrite the energy in terms of the new kernel $\overline{K}$ and we give an alternative expression for the energy of doubly radial odd functions. Finally, we establish some results that are useful when using variational techniques, and that will be exploited in the next section.

Let us start by defining the functional spaces that we are going to consider in the rest of the paper. Given a set $\Omega \subset \R^n$ and a translation invariant and positive kernel $K$ satisfying \eqref{Eq:Symmetry&IntegrabilityOfK}, we define the Hilbert space
$$
\H^K(\Omega) := \setcond{w \in L^2(\Omega)}{[w]^2_{\H^K(\Omega)} < + \infty},
$$
where
$$
[w]^2_{\H^K(\Omega)} := \dfrac{1}{2}\int\int_{(\R^{n})^2 \setminus (\R^n\setminus\Omega)^2} |w(x) - w(y)|^2 K(x-y) \d x \d y\,.
$$
We also define
\begin{align*}
\H^K_0(\Omega) &:= \setcond{w \in \H^K(\Omega)}{ w = 0 \quad \textrm{a.e. in } \R^n \setminus \Omega} \\
&\ = \setcond{w \in \H^K(\R^n)}{ w = 0 \quad \textrm{a.e. in } \R^n \setminus \Omega}.
\end{align*}

Assume that $\Omega \subset \R^{2m}$ is a set of double revolution. Then, we define
$$
\widetilde{\H}^K(\Omega) := \setcond{w \in \H^K(\Omega)}{w \textrm{ is doubly radial a.e.}}.
$$
and
$$
\widetilde{\H}^K_0(\Omega) := \setcond{w \in \H^K_0(\Omega)}{w \textrm{ is doubly radial a.e.}}.
$$
We will add the subscript `odd' and `even' to these spaces to consider only functions that are odd (respectively even) with respect to the Simons cone.

\begin{remark}
	\label{Remark:DecompositionHK}
	If $\widetilde{\H}^K_0(\Omega)$ is equipped with the scalar product
	$$
	\langle v,w \rangle_{\widetilde{\H}^K_0(\Omega)} := \dfrac{1}{2}\int_{\R^{2m}} \int_{\R^{2m}}  \big(v(x) - v(y)\big)\big(w(x) - w(y)\big) K(x-y) \d x \d y\,,
	$$
	then it is easy to check that $\widetilde{\H}^K_0(\Omega)$ can be decomposed as the orthogonal direct sum of $\widetilde{\H}^K_{0,\, \mathrm{even}}(\Omega)$ and $\widetilde{\H}^K_{0,\,\mathrm{odd}}(\Omega)$.
\end{remark}

Note that when $K$ satisfies \eqref{Eq:Ellipticity}, then $\H^K_0 (\Omega) = \H^\s_0 (\Omega)$,
which is the space associated to the kernel of the fractional Laplacian, $K(z) = c_{n,\s} |z|^{-n-2\s}$.
Furthermore, $\H^\s(\Omega) \subset H^\s(\Omega)$, where $H^\s(\Omega)$ is the usual fractional Sobolev space where interactions of $x$ and $y$ are only computed in $\Omega \times \Omega$ (see \cite{HitchhikerGuide}).  For more comments on this, see~\cite{CozziPassalacqua}, and the references therein.

Once presented the functional setting of our problem, we proceed with the study of the energy functional associated to equation \eqref{Eq:NonlocalAllenCahn}.

Given a kernel $K$ satisfying \eqref{Eq:Symmetry&IntegrabilityOfK}, a potential $G$, and a function $w\in \H^K(\Omega)$, with $\Omega\subset \R^{n}$, we write the energy defined in \eqref{Eq:Energy} as
$$
\ecal(w, \Omega) = \ecal_\mathrm{K}(w,\Omega) + \ecal_\mathrm{P}(w,\Omega)\,,
$$
where
$$
\ecal_\mathrm{K}(w, \Omega) := \dfrac{1}{2} [w]^2_{\H^K(\Omega)} \quad \text{ and } \quad  \ecal_\mathrm{P}(w, \Omega) := \int_{\Omega} G(w) \d x
\,.
$$
We will call $\ecal_\mathrm{K}$ and $\ecal_\mathrm{P}$ the \emph{kinetic} and \emph{potential} energies respectively. Recall that sometimes it is useful to rewrite the kinetic energy as
\begin{equation}
\label{Eq:KineticEnergyInteractions}
\begin{split}
\ecal_\mathrm{K}(w, \Omega) = \dfrac{1}{4} \left \{ \int_\Omega \int_\Omega |w(x) - w(y)|^2 K(x-y) \d x \d y \right. \qquad \qquad \\
+\left. 2 \int_\Omega \int_{\R^n \setminus \Omega} |w(x) - w(y)|^2 K(x-y) \d x \d y \right \}.
\end{split}	
\end{equation}
Roughly speaking, we have $\ecal_\mathrm{K}$ split into two parts: ``interactions inside-inside'' and ``interactions inside-outside''.

Note that, for functions $w\in \H^K_0(\Omega)$, it holds $\ecal_\mathrm{K}(w,\Omega) = \ecal_\mathrm{K}(w,\R^n)$. Moreover, if $G\geq 0$, the energy satisfies $\ecal(w, \Omega) \leq \ecal(w, \Omega')$  whenever $ \Omega \subset \Omega'$.

Our goal is to rewrite the kinetic energy of doubly radial odd functions in terms of the kernel $\overline{K}$ and with integrals computed only in $\ocal$, in the spirit of the previous section with the operator $L_K$. In particular, we are interested in finding an expression similar to \eqref{Eq:KineticEnergyInteractions}, where the positive kernel $\overline{K}(x,y) - \overline{K}(x,y^\star)$ appears. To do this, we introduce the following notation for the interaction. Given $A$, $B\subset \ocal$ sets of double revolution, we define
\begin{equation}
\label{Eq:DefIw}
\begin{split}
I_w(A,B) := 2\int_A  \int_B  \ |w(x)-w(y)|^2 \left\{ \overline{K}(x,y) - \overline{K}(x,y^\star) \right\} \d x \d y  \\
+\, 4 \int_A  \int_B  \left\{w^2(x)+w^2(y)\right\} \overline{K}(x,y^\star) \d x \d y\,.
\end{split}
\end{equation}
Thanks to this notation, we rewrite the kinetic energy as follows.

\begin{lemma}
	\label{Lemma:ShortExpressionEnergy}
	Let $\Omega_1, \Omega_2 \subset \R^{2m}$ be two sets of double revolution that are symmetric with respect to the Simons cone, i.e., $\Omega_i^\star = \Omega_i$, and let $w\in \widetilde{\H}^K_{0,\, \mathrm{odd}}(\R^n)$. Let $K$ be a radially symmetric kernel satisfying \eqref{Eq:Symmetry&IntegrabilityOfK}. Then, 
	\begin{equation}
	\label{Eq:InteractionsEquality}
	\int_{\Omega_1} \int_{\Omega_2} |w(x)-w(y)|^2 K(x-y) = I_w(\Omega_1\cap \ocal, \Omega_2\cap \ocal),
	\end{equation}
	where $I_w(\cdot, \cdot)$ is the interaction defined in \eqref{Eq:DefIw}.
	
	As a consequence, given a doubly radial set $\Omega\subset \R^{2m}$ with $\Omega^\star = \Omega$, and a function $v\in \widetilde{\H}^K_{0,\, \mathrm{odd}}(\Omega)$, we can write the kinetic energy as
	\begin{align}
	\label{Eq:ShortExpressionEnergy}
	\ecal_\mathrm{K}(v, \Omega) = \frac{1}{4} \big \{I_v(\Omega\cap\ocal,\Omega\cap\ocal) +  2I_v(\Omega\cap\ocal,\ocal\setminus\Omega) \big \}.
	\end{align}
	
\end{lemma}

\begin{proof}
	First, note that equality \eqref{Eq:ShortExpressionEnergy} for the kinetic energy follows directly combining expressions \eqref{Eq:KineticEnergyInteractions} and \eqref{Eq:InteractionsEquality}. Hence, we only need to prove \eqref{Eq:InteractionsEquality}.
	
	Now, since $w$ is doubly radial and $\Omega_1, \Omega_2$ are sets of double revolution, we obtain 
	$$
	\int_{\Omega_1} \int_{\Omega_2} |w(x)-w(y)|^2 K(x-y) \d x \d y = \int_{\Omega_1} \int_{\Omega_2} |w(x)-w(y)|^2 \overline{K}(x,y) \d x \d y\,,
	$$
	once we consider the change $ y = R\tilde{y}$ and take the average among all $R\in O(m)^2$ as in Lemma~\ref{Lemma:AlternativeOperatorExpression}.
	
	Finally, we split $\Omega_i$ into $\Omega_i \cap \ocal$ and $\Omega_i\setminus\ocal = (\Omega_i \cap \overline{\ocal})^\star$. By using the change of variables given by $(\cdot)^\star$ and the symmetries of $\Omega_i$ and $w$, we get
	\begin{align*}
	\int_{\Omega_1} \int_{\Omega_2} |w(x)&-w(y)|^2 \overline{K}(x,y) \d x \d y \\
	= & \hspace{1.3mm} 2 \int_{\Omega_1\cap \ocal} \int_{\Omega_2 \cap \ocal} |w(x)-w(y)|^2 \overline{K}(x,y) + |w(x)+w(y)|^2 \overline{K}(x,y^\star) \d x \d y  \\
	= & \hspace{1.3mm} 2 \int_{\Omega_1\cap \ocal} \int_{\Omega_2 \cap \ocal}  \ |w(x)-w(y)|^2 \left\{ \overline{K}(x,y) - \overline{K}(x,y^\star) \right\} \d x \d y \\
	&+4 \int_{\Omega_1\cap \ocal} \int_{\Omega_2\cap \ocal} \left\{w^2(x)+w^2(y)\right\} \overline{K}(x,y^\star) \d x \d y\,\\
	= & \hspace{1.3mm}  I_w(\Omega_1\cap\ocal,\Omega_2\cap\ocal).
	\end{align*}
	Here we have used that $\overline{K}(x^\star,y^\star) = \overline{K} (x,y)$ ---see Lemma~\ref{Lemma:PropertiesStar}.
\end{proof}

Using the previous expression for the energy, we can establish now the following lemma regarding the decrease of the energy under some operations. This result will be crucial in the next section, since it will allow us to assume that the minimizers of the energy are bounded by $1$ by above (respectively by $-1$ by below) and that are nonnegative in $\ocal$.
\begin{lemma}
	\label{Lemma:DecreaseEnergy} 
	Let $\Omega\subset \R^{2m}$ be a set of double revolution that is symmetric with respect to the Simons cone, and let $K$ be a radially symmetric kernel satisfying the positivity condition \eqref{Eq:KernelInequality}. Given $u\in\widetilde{\H}^K_{0,\mathrm{odd}}(\Omega)$, we define
	\begin{equation*}
	v(x) = \begin{cases}
	\hspace{3.2mm}|u(x)| \,\,\, &\text{if } \,\,\, x\in\ocal,\\
	-|u(x)| \,\,\, &\text{if } \,\,\, x\in\ical\, ,
	\end{cases}
	\quad 
	\text{ and }
	\quad
	w(x) = \begin{cases}
	\hspace{3.6mm}\min\{1,u(x)\} \,\,\, &\text{if } \,\,\, x\in\ocal,\\
	\,\,\,\max\{-1,u(x)\} \,\,\, &\text{if } \,\,\, x\in\ical\,.
	\end{cases}
	\end{equation*}
	
	If $G$ satisfies \eqref{Eq:HipothesesG}, then
	$$ \ecal(v,\Omega) \leq \ecal(u,\Omega) \quad 
	\text{ and }
	\quad \ecal(w,\Omega) \leq \ecal(u,\Omega) \,.  $$
\end{lemma}

\begin{proof}
	We first establish the result for $v$. Let us show that  $\ecal_\mathrm{K}(v) \leq \ecal_\mathrm{K}(u)$. Note that $v\in \widetilde{\H}^K_{0,\mathrm{odd}}(\Omega)$. Thus, by using the expression of the kinetic energy given in \eqref{Eq:ShortExpressionEnergy} and the fact that $\overline{K}(x,y) > \overline{K}(x,y^\star)> 0$ if $x,y\in \ocal$ ---see \eqref{Eq:KernelInequality}---, we only need to check that $|v(x)-v(y)|^2\leq |u(x)-u(y)|^2$ and $v^2(x)\leq u^2(x)$ whenever $x,y\in\ocal$. The first condition follows from the equivalence
	$$ \big||u(x)|-|u(y)|\big|^2\leq |u(x)-u(y)|^2 \Longleftrightarrow u(x)u(y) \leq |u(x)u(y)|,  
	$$
	while the second one is trivial and it is in fact an equality. Concerning the potential energy, since $G$ is an even function we have that $\ecal_\mathrm{P}(v) = \ecal_\mathrm{P}(u)$, and therefore we get the desired result for $v$ by adding the kinetic and potential energies.
	
	We show now the result for $w$. Let us show that  $\ecal_\mathrm{K}(w) \leq \ecal_\mathrm{K}(u)$. As before, $w\in \widetilde{\H}^K_{0,\mathrm{odd}}(\Omega)$ and thus, in view of \eqref{Eq:ShortExpressionEnergy} and the kernel inequality \eqref{Eq:KernelInequality}, we only need to check that $|w(x)-w(y)|^2\leq |u(x)-u(y)|^2$ and $w^2(x) \leq u^2(x)$ whenever $x,y\in\ocal$. The first inequality is trivial whenever $u(x)\leq 1$ and $u(y)\leq 1$, or $u(x)\geq 1$ and $u(y)\geq 1$. If $u(x)\geq 1$ and $u(y)\leq 1$, then $ |u(x)-u(y)|^2-|w(x)-w(y)|^2 = |u(x)-u(y)|^2-|1-u(y)|^2 = (u(x)-1))^2+2(u(x)-1)(1-u(y)) \geq 0$. The second inequality follows from the fact that $w^2(x) = u^2(x)$ when $u(x)\leq 1$, while $w^2(x) = 1 \leq u^2(x)$ if $u(x)\geq 1$. Concerning the potential energy, since $G$ is such that $G(x)\geq G(1) = G(-1) = 0$ if $|x|\leq 1$, then clearly $\ecal_\mathrm{P}(w) \leq \ecal_\mathrm{P}(u)$, and therefore we get the desired result by adding the kinetic and potential energies.
\end{proof}

Next we present a result that will be used later, and concerns weak solutions to semilinear Dirichlet problems. Its main consequence is that a function $u\in \widetilde{\H}^K_{0}(\Omega)$ that minimizes the energy $\ecal$, but only among doubly radial functions, is actually a weak solution to a semilinear Dirichlet problem in $\Omega$. We remark that to show the following result we do not need to use the kernel $\overline{K}$.

\begin{proposition}
	\label{Prop:WeakSolutionForAllTestFunctions}
	Let $\Omega \subset \R^{2m}$ be a bounded set of double revolution and let $L_K \in \lcal_0$ with kernel $K$ radially symmetric. Let $u\in \widetilde{\H}^K_{0}(\Omega)$ be such that
	$$
	\int_{\R^{2m}}\int_{\R^{2m}} \{u(x)-u(y)\}\{\xi(x)-\xi(y)\} K(|x-y|) \d x \d y = \int_{\R^{2m}} f(u(x)) \xi(x) \d x
	$$
	for every $\xi \in C^\infty_c(\Omega)$ that is doubly radial. Then, $u$ is a weak solution to
	\begin{equation}
		\label{Eq:DirichletProblemOmega}
	\beqc{\PDEsystem}
	L_K u &=& f(u) & \text{in } \Omega\,,\\
	u &=& 0 & \text{in } \R^{2m}\setminus \Omega\,,
	\eeqc
	\end{equation}
	i.e.,
	$$
	\int_{\R^{2m}}\int_{\R^{2m}} \{u(x)-u(y)\}\{\eta(x)-\eta(y)\} K(|x-y|) \d x \d y = \int_{\R^{2m}} f(u(x)) \eta(x) \d x
	$$
	for every $\eta \in C^\infty_c(\Omega)$ (not necessarily doubly radial).
	
	As a consequence, if $u\in \widetilde{\H}^K_{0}(\Omega)$ is a doubly radial odd minimizer of the energy $\ecal(u,\Omega)$, then it is a weak solution to \eqref{Eq:DirichletProblemOmega}.
\end{proposition}

\begin{proof}
	Let $\eta \in C^\infty_c(\Omega)$. We define an associated doubly radial function by
	$$
	\overline{\eta}(x) := \average_{O(m)^2}\eta(R x)\d R\,.
	$$
	
	Now, on the one hand, given $R\in O(m)^2$ and using the change $x = R\tilde{x}$, $y = R \tilde{y}$ and the fact that $u$ is doubly radial, we get
	\begin{align*}
	&\int_{\R^{2m}}\int_{\R^{2m}} \{u(x)-u(y)\}\{\eta(x)-\eta(y)\} K(|x-y|) \d x \d y = \\
	&\quad \quad \quad = \int_{\R^{2m}}\int_{\R^{2m}} \{u(x)-u(y)\}\{\eta(R x)-\eta(R y)\} K(|x-y|) \d x \d y\,.
	\end{align*}
	Taking the average in the previous equality among all $R\in O(m)^2$ we obtain
	\begin{align*}
	& \int_{\R^{2m}}\int_{\R^{2m}} \{u(x)-u(y)\}\{\eta(x)-\eta(y)\} K(|x-y|) \d x \d y = \\
	&\quad \quad \quad =\average_{O(m)^2} \int_{\R^{2m}}\int_{\R^{2m}} \{u(x)-u(y)\}\{\eta(R x)-\eta(R y)\} K(|x-y|) \d x \d y \d R \\
	&\quad \quad \quad= \int_{\R^{2m}}\int_{\R^{2m}} \{u(x)-u(y)\}\left \{\overline{\eta}(x) -\overline{\eta}(y)  \right \} K(|x-y|) \d x \d y \,.
	\end{align*}
	
	On the other hand, using also the change $x = R\tilde{x}$, we have
	$$
	\int_{\Omega} f(u(x)) \eta(x) \d x = \int_{\Omega} f(u(R^{-1}x)) \eta(x) \d x = \int_{\Omega} f(u(x)) \eta(Rx) \d x\,.
	$$
	Similarly as before, taking the average among all $R\in O(m)^2$, we get
	$$
	\int_{\Omega} f(u(x)) \eta(x) \d x = \average_{O(m)^2} \int_{\Omega} f(u(x)) \eta(Rx) \d x \d R = \int_{\Omega} f(u(x))\overline{\eta}(x) \d x\,.
	$$
	
	Hence, since $\overline{\eta} \in C^\infty_c(\Omega)$ is doubly radial, we have
	\begin{align*}
	&\int_{\R^{2m}}\int_{\R^{2m}} \{u(x)-u(y)\}\{\eta(x)-\eta(y)\} K(|x-y|) \d x \d y - \int_{\Omega} f(u(x)) \eta(x) \d x \\
	&\quad \quad= \int_{\R^{2m}}\int_{\R^{2m}} \{u(x)-u(y)\}\left \{\overline{\eta}(x) -\overline{\eta}(y)  \right \} K(|x-y|) \d x \d y - \int_{\Omega} f(u(x))\overline{\eta}(x) \d x \\
	&\quad \quad= 0\,,
	\end{align*}
	and thus the first result is proved.

	We next show that if $u$ is a doubly radial odd minimizer, then it is a weak solution to \eqref{Eq:DirichletProblemOmega}. To see this, we consider perturbations $u +  \varepsilon \xi$ with $\varepsilon\in \R$ and $\xi \in \widetilde{\H}^K_{0}(\Omega)$. By Remark~\ref{Remark:DecompositionHK}, it suffices to consider only even and odd functions $\xi$. Let first $\xi \in \widetilde{\H}^K_{0, \,\mathrm{odd}}(\Omega)$. Then, since $u$ is a minimizer among functions in $\widetilde{\H}^K_{0, \,\mathrm{odd}}(\Omega)$, we get
	$$
	0 = \dfrac{\d}{\d \varepsilon}\evalat{\varepsilon = 0} \ecal(u +  \varepsilon \xi, \Omega) = \langle u,\xi \rangle_{\widetilde{\H}^K_0(\Omega)} - \langle f(u),\xi \rangle_{L^2(\Omega)}\,.
	$$
	Next, take $\xi \in \widetilde{\H}^K_{0, \,\mathrm{even}}(\Omega)$. Since $u$ is odd with respect to the Simons cone, the same holds for $f(u)$ ---recall that $f$ is odd. Thus,
	$$
	\langle u,\xi \rangle_{\widetilde{\H}^K_0(\Omega)} = 0 \quad \textrm{ and } \quad  \langle f(u),\xi \rangle_{L^2(\Omega)} = 0\,.
	$$
	Therefore, 
	$$
	\langle u,\xi \rangle_{\widetilde{\H}^K_0(\Omega)} = \langle f(u_R),\xi \rangle_{L^2(\Omega)}
	$$
	for every $\xi \in\widetilde{\H}^K_0(\Omega)$ with compact support in  $\Omega$. In particular,
	$$
	\int_{\R^{2m}}\int_{\R^{2m}} \{u_R(x)-u_R(y)\}\{\xi(x)-\xi(y)\} K(|x-y|) \d x \d y = \int_{\R^{2m}} f(u_R(x)) \xi(x) \d x
	$$
	for every $\xi \in C^\infty_c(\Omega)$ that is doubly radial. Finally, by the first statement of the proposition, that we just proved, we obtain that $u$ is a weak solution to \eqref{Eq:DirichletProblemOmega}.
\end{proof}

The previous proposition, combined with the regularity results of the following remark, yields that bounded minimizers among doubly radial functions of the energy $\ecal(\cdot,\Omega)$ are classical solutions to $L_K u = f(u)$ in $\Omega$.

\begin{remark}
	\label{Remark:InteriorRegularity}
	 Let us present here some interior estimates that will be used in the sequel. If $w\in L^\infty (\R^n)$ is a weak solution to $L_K w = h$ in $B_1\subset \R^n$, with $L_K \in \lcal_0(n,\s,\lambda, \Lambda)$, then
	\begin{equation}
	\label{Eq:C2sEstimate}
	\norm{w}_{C^{2\s} (\overline{B_{1/2}})} \leq C\bpar{\norm{h}_{L^\infty (B_1)} + \norm{w}_{L^\infty  (\R^n)}}.
	\end{equation} 
	If, in addition, $w \in C^\alpha (\R^n)$ with $\alpha + 2\s$ not an integer, then
	\begin{equation}
	\label{Eq:Calpha->Calpha+2sEstimate}
	\norm{w}_{C^{\alpha + 2\s} (\overline{B_{1/2}})} \leq C\bpar{\norm{h}_{C^{\alpha} (\overline{B_1})} + \norm{w}_{C^\alpha (\R^n)} },
	\end{equation}
	where the previous two constants $C$ depend only on $n$, $\s$, $\lambda$, and $\Lambda$ (see \cite{RosOton-Survey,SerraC2s+alphaRegularity} and the references therein).

	Note that in some situations these estimates are not suitable enough to be applied repeatedly due to the term $\norm{w}_{C^\alpha (\R^n)}$ in \eqref{Eq:Calpha->Calpha+2sEstimate}. 
	Let us show how to overcome this difficulty in our setting, that is, when $K$ is radially symmetric and $K(\sqrt{\cdot})$ is convex. 
	In this case, using the ellipticity property \eqref{Eq:Ellipticity} and the convexity assumption \eqref{Eq:SqrtConvex} for $K$, it is not difficult to show that $K$ is locally Lipschitz in $\R^n \setminus \{0\}$ and
	\begin{equation}
	\label{Eq:KLipschitz}
		\seminorm{K}_{\Lip (\R^n \setminus B_R)} \leq C R^{-n-2\s-1}, \quad \text{ for all } R>0,
	\end{equation} 
	with a positive constant $C$ depending only on $n$, $\s$, $\lambda$, and $\Lambda$.
	
	Using for $L_K$ the same cut-off argument as in Corollary~2.4 of \cite{RosOtonSerra-Regularity} for the fractional Laplacian, and taking into account \eqref{Eq:KLipschitz}, one can modify \eqref{Eq:Calpha->Calpha+2sEstimate} to obtain the estimate 
	\begin{equation} 
	\label{Eq:Calpha->Calpha+2sEstimateBalls} 
	\norm{w}_{C^{\alpha + 2\s} (\overline{B_{1/4}})} \leq C\bpar{\norm{h}_{C^{\alpha} (\overline{B_1})} + \norm{w}_{C^\alpha (\overline{B_1})} + \norm{\dfrac{w(x)}{(1+|x|)^{n+2\s}}}_{L^1(\R^n)} }, 
	\end{equation} 
	for all $\alpha \in (0,1)$ with $\alpha + 2\s$ not an integer, and with $C$ depending only on $n$, $\s$, $\lambda$, and $\Lambda$.
	This, combined with \eqref{Eq:C2sEstimate}, will be used in the following section to obtain uniform Lipschitz interior estimates for the semilinear equation $L_K u = f(u)$.	
\end{remark}

\section{An energy estimate for doubly radial odd minimizers}
\label{Sec:EnergyEstimate}

In this section we present an estimate for the energy in the ball $B_S$ of minimizers in the space $\widetilde{\H}^K_{0, \mathrm{odd}}(B_R)$ with $R > S+ 4$. That is, we prove Theorem~\ref{Th:EnergyEstimate}. In order to establish this result, we follow the ideas of Savin and Valdinoci in \cite{SavinValdinoci-EnergyEstimate}, where they show the same estimate but for minimizers without any restriction on their symmetry.

First of all, let us comment briefly the strategy used in \cite{SavinValdinoci-EnergyEstimate}. The argument is based on comparing the energy of the minimizer $u$ in $B_R\subset \R^n$ with the energy of a suitable competitor $v$. This function $v$ satisfies, in $B_{S+2}\subset B_R \subset \R^n$, the following properties:
\begin{enumerate}[label=(\textit{\roman*})]
	\item $-1 \leq v \leq 1$.
	\item $v=u$ in $\partial B_{S+2}$.
	\item The set $\{v\not \equiv -1\}\cap B_{S+2}$ has measure bounded by $C S^{n-1}$ for some constant $C$.
	\item $v\in \Lip(\overline{B_{S+2}})$ with a Lipschitz constant independent of $R$ and $S$.
\end{enumerate} 
By the second property, $v$ can be extended to coincide with $u$ outside $B_{S+2}$, becoming an admissible competitor. Then, the desired estimate follows by finding precise bounds on the energy of $v$ in $B_{S+2}$. The function $v$ is constructed in $B_{S+2}$ as $v = \min \{u, \phi_S\}$, where $\phi_S (x) =-1+2\min\{(|x|-S-1)_+,1\}$ ---we will also use this function below, see \eqref{Eq:DefOfPhiS}.

In our case, the strategy will be the same but adapting some ingredients, namely, the competitor $v$. First, note that the previous construction for $v$ cannot be used in our setting, since it would not produce a doubly radial odd function. To overcome this problem, we will construct a function $w$ defined in $B_{S+2}\cap \ocal$ and satisfying the four previous assumptions on $v$. In addition, we will require $w$ to be doubly radial and to vanish on the Simons cone (then we will consider its odd extension through $\ccal$).

To state the precise properties of $w$, we need to consider the Lipschitz constant of $u$ in $\overline{B_{S+3}}$, namely
\begin{equation} 
\label{Eq:ChoiceMu} 
	\mu := \seminorm{u}_{\mathrm{Lip}(\overline{B_{S+3}})}. 
\end{equation} 
By Proposition~\ref{Prop:WeakSolutionForAllTestFunctions}, we know that $u$ solves $L_K u = f(u)$ in $B_R$ with $R> S+4$. Moreover, by Lemma~\ref{Lemma:DecreaseEnergy} we know that $u$ is bounded. Therefore, by applying repeatedly the estimates \eqref{Eq:C2sEstimate} and \eqref{Eq:Calpha->Calpha+2sEstimateBalls} in balls centered at points in $B_{S+3}$, it is easy to see that $\mu\leq C$ with a positive constant $C$ depending only on $m$, $\s$, $\lambda$, $\Lambda$, and $\norm{f}_{C^1([-1,1])}$ (and thus, independent of $R$ and $S$). Recall that $G' = -f$ and hence $\norm{f}_{C^1([-1,1])} \leq \norm{G}_{C^2([-1,1])}$.

We can now define the set
\begin{equation}
\label{Eq:DefOmegaS}
\Omega_S := \left( \overline{B_{S+2}}\setminus B_S \right) \cup \left(  \overline{B_{S+2}} \cap \{\mu \dist(\cdot,\ccal) \leq 1\}\right),
\end{equation} 
---see Figure~\ref{Fig:PsiSandOmegaS}~(a). It is easy to see that 
\begin{equation}
\label{Eq:MeasureOmegaS}
|\Omega_S| \leq C\,S^{2m-1},
\end{equation}
with a constant $C$ depending only on $m$ and $\mu$. This can be checked following the computations in the proof of the energy estimate for the local equation in Theorem~1.3 of \cite{CabreTerraI}. 

\begin{figure}
	\centering
	\hspace{-0.26\textwidth} 
	\begin{subfigure}{0.21\textwidth}
		\centering
		\definecolor{azul_custom}{RGB}{66,240,209}
\definecolor{lila_custom}{RGB}{201,69,254}
\definecolor{naranja_custom}{RGB}{255,148,0}

\begin{tikzpicture}[y=0.80pt, x=0.80pt, yscale=-1.000000, xscale=1.000000, inner sep=0pt, outer sep=0pt]

\path[fill=azul_custom,line cap=round,miter limit=4.00,line width=1.216pt]
(210.3958,249.7906) .. controls (188.0101,272.1763) and (172.1936,287.9233) ..
(150.1544,309.9625) .. controls (148.2221,309.9625) and (137.7690,309.7049) ..
(127.0089,309.7371) .. controls (127.0089,300.7477) and (127.0826,301.1491) ..
(127.0826,287.3606) .. controls (136.2973,278.1061) and (177.3049,237.1011) ..
(187.7910,226.5786) .. controls (170.0299,214.0930) and (149.7638,207.7938) ..
(127.0312,207.7938) .. controls (127.0312,190.9391) and (126.9993,161.0751) ..
(126.9993,146.8284) .. controls (170.9394,146.8284) and (208.6039,162.4681) ..
(243.0000,193.8560) .. controls (271.9385,225.9716) and (289.9164,263.3638) ..
(289.9164,309.8622) .. controls (260.7310,309.8622) and (257.9656,309.8622) ..
(229.1006,309.8622) .. controls (229.1006,279.3951) and (221.7362,269.4328) ..
(210.3958,249.7907) -- cycle;

\path[draw=black,line join=miter,line cap=butt,line width=1.5pt]
(126.5055,309.8063) -- (330.4872,309.8063);

\path[draw=black,fill=black,even odd rule,line width=0.497pt]
(330.4872,309.8063) -- (328.1493,312.1289) -- (336.3320,309.8063) --
(328.1493,307.4838) -- cycle;

\path[draw=black,line join=miter,line cap=butt,line width=1.5pt]
(127.0562,310.75) -- (127.0562,99.5223);

\path[draw=black,fill=black,even odd rule,line width=0.497pt] (127.0562,99.5223)
-- (129.3788,101.8602) -- (127.0562,93.6775) -- (124.7336,101.8602) -- cycle;

\path[draw=black,line join=miter,line cap=butt,line width=1.4pt]
(127.0312,146.9770) .. controls (166.9982,146.9770) and (210.2478,160.9750) ..
(243.0000,193.8560) .. controls (275.7522,226.7370) and (289.9164,269.7115) ..
(289.9164,309.8622);

\path[draw=black,line join=miter,line cap=butt,line width=1.4pt]
(127.0312,177.3949) .. controls (151.7088,177.3949) and (192.2090,185.8302) ..
(221.4938,215.3997) .. controls (250.7785,244.9692) and (259.4986,284.9193) ..
(259.4986,309.8622);

\path[draw=black,line join=miter,line cap=butt,line width=1.4pt]
(127.0312,207.7933) .. controls (150.2506,207.7933) and (177.3815,214.3180) ..
(199.9500,236.9435) .. controls (222.0483,259.0975) and (229.1002,286.5606) ..
(229.1002,309.8622);

\path[draw=black,line join=miter,line cap=butt,miter limit=4.00,even odd
rule,line width=1.57pt] (126.7629,310.0643) -- (326.1869,110.8844);

\path[draw=black,line join=miter,line cap=butt,even odd rule,line width=0.704pt]
(150.2563,309.8481) -- (318.0555,142.0518);

\path[draw=black,line join=miter,line cap=butt,even odd rule,line width=0.704pt]
(127.1847,287.2416) -- (299.4946,114.9526); 

\path[draw=black,line join=miter,line cap=butt,line width=0.5pt]
(298.3343,142.4919) -- (306.3936,150.7321);

\path[draw=black,fill=black,even odd rule,line width=0.200pt]
(298.9657,143.1375) -- (300.2914,143.1522) -- (297.3269,141.4619) --
(298.9510,144.4632) -- cycle;

\path[draw=black,fill=black,even odd rule,line width=0.200pt]
(305.7622,150.0865) -- (304.4364,150.0718) -- (307.4010,151.7621) --
(305.7769,148.7608) -- cycle;

\node at (315,138) {\normalsize $\mu^{-1}$};
\node at (127,85) {\normalsize $|x''|$};
\node at (350, 311) {\normalsize $|x'|$};
\node at (228,320) {\normalsize $S$};
\node at (255, 320) {\normalsize $S\!+\!1$};
\node at (292, 320) {\normalsize $S\!+\!2$};
\node at (334, 101) {\normalsize $\ccal$};
\node at (264, 250) {\normalsize $\Omega_S$};

\end{tikzpicture}
	\end{subfigure}
	\hspace{0.28\textwidth} 
	\begin{subfigure}{0.21\textwidth}
		\centering		
		\definecolor{azul_custom}{RGB}{66,240,209}
\definecolor{lila_custom}{RGB}{201,69,254}
\definecolor{naranja_custom}{RGB}{255,148,0}

\begin{tikzpicture}[y=0.80pt, x=0.80pt, yscale=-1.000000, xscale=1.000000, inner sep=0pt, outer sep=0pt]

\path[scale=0.938,fill=lila_custom, opacity = 0.7, line width=0.400pt] (248.0459,242.6193) ..
controls (275.7060,284.1286) and (276.2013,310.5581) .. (276.7985,330.5197) ..
controls (233.8448,330.5197) and (216.3370,330.5469) .. (160.2334,330.5469) ..
controls (186.2822,304.6380) and (219.4410,271.2569) .. (248.0459,242.6193) --
cycle;

\path[draw=black,line join=miter,line cap=butt,line width=1.4pt]
(243.0000,193.8560) .. controls (275.7522,226.7370) and (289.9164,269.7115) ..
(289.9164,309.8622);

\path[draw=naranja_custom,line join=miter,line cap=butt,line width=1.8pt]
(243.0000+11,193.8560+12) .. controls (275.7522-9,226.7370-6) and (289.9164,258.7115) ..
(289.9164,309.8622);

\path[draw=black,line join=miter,line cap=butt,line width=1.4pt]
(221.4938,215.3997) .. controls (250.7785,244.9692) and (259.4986,284.9193) ..
(259.4986,309.8622);

\path[draw=black,line join=miter,line cap=butt,line width=1.4pt]
(199.9500,236.9435) .. controls (222.0483,259.0975) and (229.1002,286.5606) ..
(229.1002,309.8622);

\path[draw=black,line join=miter,line cap=butt,miter limit=4.00,even odd
rule,line width=1.57pt] (126.7629,310.0643) -- (326.1869,110.8844);

\path[draw=black,line join=miter,line cap=butt,even odd rule,line width=0.704pt]
(150.2563,309.8481) -- (318.0555,142.0518);

\path[draw=black,line join=miter,line cap=butt,line width=0.5pt]
(298.3343,142.4919) -- (306.3936,150.7321);

\path[draw=black,fill=black,even odd rule,line width=0.200pt]
(298.9657,143.1375) -- (300.2914,143.1522) -- (297.3269,141.4619) --
(298.9510,144.4632) -- cycle;

\path[draw=black,fill=black,even odd rule,line width=0.200pt]
(305.7622,150.0865) -- (304.4364,150.0718) -- (307.4010,151.7621) --
(305.7769,148.7608) -- cycle;

\path[draw=black,line join=miter,line cap=butt,line width=0.5pt]
(300, 230) .. controls (300, 235) and (295, 250) ..
(285, 250);

\path[draw=black,fill=black,even odd rule,line width=0.200pt]
(285, 250) -- (285+0.9375, 250+0.9375) -- (285-2.3437, 250) --
(285+0.9375, 250-0.9375) -- cycle;

\path[draw=black,fill=black,even odd rule,line width=0.200pt]
(210.4938,215.3997) -- (210.4938-1.3258,215.3997-0.0147) -- (210.4938+1.6388,215.3997+1.6759) --
(210.4938+0.0147,215.3997-1.3257) -- cycle;

\path[draw=black,line join=miter,line cap=butt,line width=0.5pt]
(210.4938,215.3997) -- (210.4938-10,215.3997-10);

\path[draw=azul_custom,line join=miter,line cap=butt,line width=1.8pt]
(126.7629+1,310.0643-1) -- (243.0000+0.7,193.8560-0.7);

\path[draw=black,line join=miter,line cap=butt,line width=1.5pt]
(126.5055,309.8063) -- (330.4872,309.8063);

\path[draw=black,fill=black,even odd rule,line width=0.497pt]
(330.4872,309.8063) -- (328.1493,312.1289) -- (336.3320,309.8063) --
(328.1493,307.4838) -- cycle;

\path[draw=black,line join=miter,line cap=butt,line width=1.5pt]
(127.0562,310.75) -- (127.0562,99.5223);

\path[draw=black,fill=black,even odd rule,line width=0.497pt] (127.0562,99.5223)
-- (129.3788,101.8602) -- (127.0562,93.6775) -- (124.7336,101.8602) -- cycle;

\node at (315,138) {\normalsize $\mu^{-1}$};
\node at (127,85) {\normalsize $|x''|$};
\node at (350, 311) {\normalsize $|x'|$};
\node at (228,320) {\normalsize $S$};
\node at (255, 320) {\normalsize $S\!+\!1$};
\node at (292, 320) {\normalsize $S\!+\!2$};
\node at (334, 101) {\normalsize $\ccal$};
\node at (298, 220) {\normalsize $\Psi_S\! = \!1$};
\node at (220, 286) {\normalsize $\Psi_S\! = \!-1$};
\node at (190, 195) {\normalsize $\Psi_S\! = \!0$};

\end{tikzpicture}
	\end{subfigure}
	\caption{(a) The set $\Omega_S$. (b) The $1$ and $-1$ level sets of $\Psi_S$ in $\overline{B_{S+2}}\cap \ocal$.}
	\label{Fig:PsiSandOmegaS}
\end{figure}
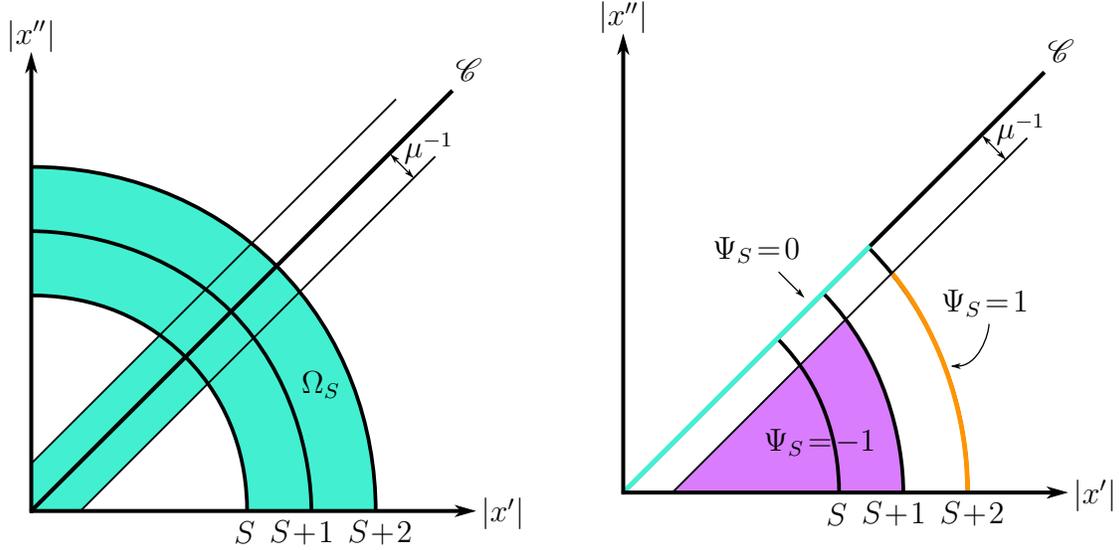

In the following lemma we state the precise properties for the competitor $w$ that suffice to establish the energy estimate given by the right-hand side of \eqref{Eq:EnergyEstimate} for $\ecal(w,B_{S+2})$.

\begin{lemma}
	\label{Lemma:ExistenceCompetitor}
	Let $S\geq 2$ and $R > S + 4$. Let $u\in \widetilde{\H}^K_{0, \mathrm{odd}}(B_R)$ be a doubly radial odd minimizer of the energy \eqref{Eq:Energy} and let $\mu$ be defined by \eqref{Eq:ChoiceMu}. Then, there exists  a function $w:\overline{ B_{S+2}\cap \ocal} \to \R$ satisfying the following:
	
	\begin{enumerate}[label={\normalfont (\textcolor{red}{H\arabic*})}, ref=H\arabic*]
		\item
		\label{Eq:wH1} $-1 \leq w \leq 1$.
		\item
		\label{Eq:wH2} $w$ doubly radial and $w=0$ in $\ccal$.
		\item
		\label{Eq:wH3} $w=u$ on $\partial B_{S+2} \cap \ocal$.
		\item
		\label{Eq:wH4} $w\equiv-1$ on $(B_{S+2}\cap \ocal)\setminus \Omega_S = B_S \cap \{ \mu \dist (\cdot, \ccal) > 1\}$.
		\item
		\label{Eq:wH5} $w\in \Lip(\overline{B_{S+2}})$ with a Lipschitz constant independent of $R$ and $S$. In addition, 
		\begin{equation}
			\label{Eq:wH5Lipschitz}
			|w(x)-w(y)| \leq \frac{C}{\dist(x,\ccal)}|x-y|			
		\end{equation}
		whenever $x,y \in B_{S+1}\cap \ocal$, $\mu \dist(x,\ccal)\geq 1$ and $\mu \dist(y,\ccal)\leq 1$, and with $C$ a constant independent of $R$ and $S$.
	\end{enumerate} 
\end{lemma}

\begin{proof}
	To construct the function $w$ we first define
	\begin{equation}
		\label{Eq:DefOfPhiS}
		\phi_S(x):= 
		\begin{cases}
		-1 & \textrm{ if }  \ |x| \leq S+1\,, \\
		-1 + 2 (|x| - S-1)&  \textrm{ if } \ S+1 \leq |x| \leq S+2\,, \\
		1 & \textrm{ if }  \  S+2 \leq |x|\,,
		\end{cases}
	\end{equation}
	which is the function used in \cite{SavinValdinoci-EnergyEstimate}. Now, we modify it in order to make it vanish on $\ccal$. We define
	$$
	\Psi_S(x) :=
	\begin{cases}
	\phi_S (x) \mu\dist(x,\ccal) &  \textrm{ if }  \  \mu\dist(x,\ccal) \leq 1 \,, \\
	\phi_S (x) & \textrm{ if }  \  \mu\dist(x,\ccal) \geq 1\,,
	\end{cases}
	$$
	---see Figure~\ref{Fig:PsiSandOmegaS}~ (b) for a schematic representation.
	
	With this function on hand, we construct the competitor  $w:\overline{B_{S+2} \cap \ocal} \to \R$ as
	$$
	w:= \min \{u, \Psi_S\}.
	$$
	We check next that \eqref{Eq:wH1}-\eqref{Eq:wH5} hold.
	
	First of all, recall that by Lemma~\ref{Lemma:DecreaseEnergy}, $0\leq u \leq 1$ in $\ocal$. Since $-1\leq \Psi_S\leq 1$ in $B_{S+2} \cap \ocal$, \eqref{Eq:wH1} holds trivially. Moreover, since both functions are doubly radial and vanish on $\ccal$, \eqref{Eq:wH2} follows ---recall that the distance to the cone, in $\ocal$, is the doubly radial function given by $(|x'|-|x''|)/\sqrt{2}$. The verification of \eqref{Eq:wH4} is easy, since $\Psi_S \equiv-1 \leq u$ in $(B_{S+2}\cap \ocal)\setminus \Omega_S$.
	
	Now, we check that \eqref{Eq:wH3} holds. On the one hand, if $x\in \partial B_{S+2}\cap \ocal$ and $\mu\dist(x,\ccal) \geq 1$, we have $\Psi_S (x) = \phi_S(x) = 1 \geq u(x)$, and therefore $w(x) = u(x)$. On the other hand, for $x\in \partial B_{S+2}\cap \ocal$ with $\mu\dist(x,\ccal) \leq 1$, we have $\Psi_S (x) = \mu \dist (x,\ccal)$. By \eqref{Eq:ChoiceMu},
	$$ 
	|u(y) - u(z)|\leq \mu |y-z| \quad \text{ for every } y, \ z \in \overline{B_{S+3}},
	$$ 
	and thus, by taking $y=x$ and $z\in \ccal$ to be a point realizing $\dist(x,\ccal)$, we obtain that
	$$ 
	u(x) = |u(x)|\leq \mu |x-z|  = \mu \dist (x,\ccal) = \Psi_S (x).
	$$ 
	Thus, $w(x) = u(x)$ and \eqref{Eq:wH3} holds.
	
	Finally, we verify \eqref{Eq:wH5}. Obviously, $w$ is Lipschitz in $\overline{B_{S+2}}$ since it is the minimum of two Lipschitz functions ---with Lipschitz constants depending only on $\mu$. From this it follows that \eqref{Eq:wH5Lipschitz} also holds, for a large constant $C$ depending on $\mu$, at points where $\dist(x,\ccal)\leq 2/\mu$. Finally, assume that $\dist(x,\ccal)\geq 2/\mu$. Then, by using the triangular inequality and the definition of distance to the Simons cone, we have
	$$
	|x-y| \geq 	\dist(x,\ccal) - \dist(y,\ccal) \geq \dfrac{1}{2}\dist(x,\ccal).
	$$
	From this and \eqref{Eq:wH1}, we readily deduce that \eqref{Eq:wH5Lipschitz} holds for a large constant $C$.	
\end{proof}

To estimate the energy of $w$ in $B_{S+2}$, it will be important to control the double integrals in the nonlocal energy  first in the set where $|x-y|\geq d_S(x)$, and then in $\{|x-y|\leq d_S(x)\}$, where 
$$ 
d_S (x) 
:= \min \{\dist(x, \partial B_{S+1}),\mu \dist(x,\ccal)\} \quad \text{ for } x\in B_S\,.
$$
A similar technicality was used by Savin and Valdinoci in \cite{SavinValdinoci-EnergyEstimate} with the function $\dist(x, \partial B_{S+1})$, and it is the key point to get \eqref{Eq:EnergyEstimate}. We can now establish the energy estimate of Theorem~\ref{Th:EnergyEstimate}. 

\begin{proof}[Proof of Theorem~\ref{Th:EnergyEstimate}]
	Take $w$ constructed in Lemma~\ref{Lemma:ExistenceCompetitor} and extend it oddly through $\ccal$ and then to coincide with $u$ outside $B_{S+2}$.  Hence, since $u$ is a doubly radial odd minimizer in $B_R$, and $w$ an admissible competitor, $\ecal (u, B_R) \leq \ecal (w, B_R)$. Moreover, $u\equiv w$ in $\R^{2m}\setminus B_{S+2}$, and thus it follows that
	$$ 
	\ecal (u, B_{S+2}) \leq \ecal (w, B_{S+2}).
	$$
	By the monotonicity of the energy $\ecal$ by inclusions we get
	$$ 
	\ecal (u, B_{S}) \leq \ecal (w, B_{S+2}). 
	$$
	Therefore, to obtain the desired result it remains to estimate $\ecal (w, B_{S+2})$.

	In the following inequalities, the letter $C$ will be a constant depending only on $m$, $\s$, $\lambda$, $\Lambda$, and $\norm{G}_{C^2([-1,1])}$. Recall that $\mu$ defined in \eqref{Eq:ChoiceMu} depends only on these quantities.
	
	First, note that using the upper bound for the kernel $K$ ---\eqref{Eq:Ellipticity}--- and the change of variables given by $(\cdot)^\star$, it follows that
	$$ 
	\ecal(w,B_{S+2}) \leq C \int_{B_{S+2}\cap \ocal} \d x \int_{\R^{2m}} \d y \frac{|w(x)-w(y)|^2}{|x-y|^{2m+2\s}} + \int_{B_{S+2}} G(w) \d x. 
	$$
	Now we estimate separately the potential and kinetic energies.
	
	\medskip
	
	\textbf{Estimate for the potential energy.}
	Since, $w=\pm 1$ in $B_{S+2} \setminus \Omega_S$ by \eqref{Eq:wH4}, $-1\leq w\leq 1$ by \eqref{Eq:wH1}, and $G(1) = G(-1) = 0$, it is clear that
	$$ 
	\int_{B_{S+2}} G(w) = \int_{\Omega_S} G(w) \leq C |\Omega_S| \leq C S^{2m-1}.
	$$  
	Here we have used \eqref{Eq:MeasureOmegaS}.
	
	\medskip
	
	\textbf{Estimate for the kinetic energy.}
	We split the integral in three terms, as follows.
	\begin{align*}
	\int_{B_{S+2}\cap \ocal} \d x \int_{\R^{2m}} \d y \frac{|w(x)-w(y)|^2}{|x-y|^{2m+2\s}} &= \int_{\Omega_{S}\cap \ocal} \d x \int_{\R^{2m}} \d y \ \frac{|w(x)-w(y)|^2}{|x-y|^{2m+2\s}} \\
	&\ \ + \int_{(B_S\setminus \Omega_S)\cap \ocal} \d x \int_{B_{S+1}\cap \ocal} \d y \ \frac{|w(x)-w(y)|^2}{|x-y|^{2m+2\s}} \\
	&\ \ + \int_{(B_S\setminus \Omega_S)\cap \ocal} \d x \int_{(B_{S+1}\cap \ocal)^c} \d y \ \frac{|w(x)-w(y)|^2}{|x-y|^{2m+2\s}} \\
	&=: I_1+I_2+I_3.
	\end{align*}
	Here $(\cdot)^c$ denotes the complementary set. Now we control each term separately.
	
	We estimate the first integral:
	\begin{align*}
	I_1 &= \int_{\Omega_{S}\cap \ocal} \d x \int_{\R^{2m}} \d y \  \frac{|w(x)-w(y)|^2}{|x-y|^{2m+2\s}} \\
	&= \int_{\Omega_{S}\cap \ocal} \d x \int_{\{|x-y|\leq 1\}} \d y \ \frac{|w(x)-w(y)|^2}{|x-y|^{2m+2\s}} + \int_{\Omega_{S}\cap \ocal} \d x \int_{\{|x-y|\geq 1\}} \d y \ \frac{|w(x)-w(y)|^2}{|x-y|^{2m+2\s}} \\
	&\leq C \int_{\Omega_{S}\cap \ocal} \d x \int_0^1 \d r \ r^{1-2\s} + C \int_{\Omega_{S}\cap \ocal} \d x \int_1^\infty \d r \ r^{-1-2\s} \\
	&= C |\Omega_S| \leq C \ S^{2m-1}.
	\end{align*}
	We have used that $w$ is Lipschitz in $\overline{B_{S+3}}$ ---see \eqref{Eq:wH5} and \eqref{Eq:ChoiceMu}--- to bound the first integral, while the second one is controlled using only that $w$ is bounded, by \eqref{Eq:wH1}. 
	
	Next, we estimate $I_2$. To do it, we first claim that, if $|x-y|\leq d_S(x)$, then
	\begin{equation}
	\label{Eq:EnergyEstimateProofLipschitz}
	|w(x)-w(y)| \leq \dfrac{C}{d_S(x)} |x-y|	
	\end{equation}
	for every $x\in (B_S\setminus \Omega_S)\cap \ocal$ and $y \in B_{S+1}\cap \ocal$. Recall that $d_S$ is defined as $d_S(x)=\min \{\dist(x,\partial B_{S+1}) , \mu \dist (x, \ccal)\}$, and therefore it suffices to show that
	$$
	|w(x)-w(y)| \leq \dfrac{C}{\dist (x, \ccal)} |x-y|
	$$
	for $x\in B_S \cap \ocal$ with $\mu \dist (x, \ccal) \geq 1$ and $y \in B_{S+1}\cap \ocal$ (recall that $C$ may depend on $\mu$). Now, if we assume that $\mu \dist (y, \ccal) \geq 1$, it follows that $w(x)=w(y)=-1$ and \eqref{Eq:EnergyEstimateProofLipschitz} is trivially true. On the other hand, if we assume that $\mu \dist (y, \ccal) \leq 1$, then \eqref{Eq:EnergyEstimateProofLipschitz} follows from \eqref{Eq:wH5}. Therefore, the claim is proved.
	
	Using \eqref{Eq:EnergyEstimateProofLipschitz}, we proceed as before splitting the integrals to obtain
	\begin{align*}
	I_2 &= \int_{(B_S\setminus \Omega_S)\cap \ocal} \d x \int_{B_{S+1}\cap \ocal} \d y \  \frac{|w(x)-w(y)|^2}{|x-y|^{2m+2\s}}   \\
	&\leq \int_{(B_S\setminus \Omega_S)\cap \ocal} \d x \int_{ \{|x-y|\leq d_S(x)\} \cap B_{S+1}\cap \ocal} \d y \  \frac{|w(x)-w(y)|^2}{|x-y|^{2m+2\s}} \\
	& \quad \quad + \int_{(B_S\setminus \Omega_S)\cap \ocal} \d x \int_{\{|x-y|\geq d_S(x)\}} \d y \ \frac{|w(x)-w(y)|^2}{|x-y|^{2m+2\s}} \\
	&\leq C \int_{(B_S\setminus \Omega_S)\cap \ocal} d_S(x)^{-2\s} \d x .
	\end{align*}
	Here we have used \eqref{Eq:EnergyEstimateProofLipschitz} to estimate the first term, while for the second one we have only used that $w$ is bounded, by \eqref{Eq:wH1}. The last integral for $d_S(x)^{-2\s}$ will be bounded later on.
	
	Next, we estimate $I_3$. To do it, we first claim that if $x\in (B_S\setminus \Omega_S) \cap \ocal$ and $y\in (B_{S+1}\cap \ocal)^c = \ical \cup B_{S+1}^c$, then $|x-y|\geq c d_S(x)$ for some constant $c>0$ depending only on $\mu$. Indeed, on the one hand it is clear that, if $y\in B_{S+1}^c$, then $|x-y|\geq \dist(x,\partial B_{S+1})\geq d_S(x)$. On the other hand, if $y\in \ical$, then $|x-y|\geq \dist(x,\ccal) \geq  d_S(x) / \mu$.
	
	By the previous claim, since $w$ is bounded, we obtain
	\begin{align*}
	I_3 &= \int_{(B_S\setminus \Omega_S)\cap \ocal} \d x \int_{(B_{S+1}\cap \ocal)^c} \d y \  \frac{|w(x)-w(y)|^2}{|x-y|^{2m+2\s}} \\
	&\leq \int_{(B_S\setminus \Omega_S)\cap \ocal} \d x \int_{\{|x-y|\geq C d_S(x)\}} \d y \ \frac{|w(x)-w(y)|^2}{|x-y|^{2m+2\s}} \\
	&\leq C \int_{(B_S\setminus \Omega_S)\cap \ocal}  d_S(x)^{-2\s} \d x .
	\end{align*}
	
	Now, we add up $I_1, I_2$, and $I_3$ to get
	\begin{equation}
	\label{Eq:EnergyEstimateProofLastEstimate}
	\ecal(w,B_{S+2}) \leq C \left( \int_{(B_S\setminus \Omega_S)\cap \ocal}  d_S(x)^{-2\s} \d x + S^{2m-1}\right).
	\end{equation}
	We conclude the proof by estimating the integral of $d_S(x)^{-2\s}$, as follows.
	\begin{equation}
	\label{Eq:EnergyEstimatedS}
	\begin{split}
	\int_{B_{S+2}\setminus \Omega_S} d_S(x)^{-2\s} \d x &= \int_{B_{S}\cap \{\mu \dist(x,\ccal)>1\}} \!\!\!\! \max\{\dist(x,\partial B_{S+1})^{-2\s}, (\mu\dist(x,\ccal))^{-2\s}\} \d x \\
	& \leq \int_{B_S} \left( S+1-|x| \right)^{-2\s} \d x \\
	& \quad \quad + C \int_{B_{S}\cap \{\mu \dist(x,\ccal)>1\}} \dist(x,\ccal)^{-2\s} \d x.
	\end{split}
		\end{equation}
	We next control these two integrals.

	The first integral can be estimated by using spherical coordinates and the change $\tau = r/(S+1)$. Indeed,
	\begin{align*}
	\int_{B_S} \left( S+1-|x| \right)^{-2\s} \d x & = C \int_0^S \dfrac{r^{2m-1}}{(S+1-r)^{2\s}} \d r \\
	& \leq C (S+1)^{2m - 2\s}\int_0^{1 - \frac{1}{S+1}} \dfrac{\tau^{2m-1}}{(1-\tau)^{2\s}} \d \tau\\
	&\leq C (S+1)^{2m - 2\s}\int_0^{1 - \frac{1}{S+1}} (1-\tau)^{-2\s} \d \tau \\
	& \leq \begin{cases}
	C \ S^{2m-2\s}\ \ \ \ &\textrm{if } \ \ \s\in(0,1/2),\\
	C\ S^{2m-1}\,\log S\ \ \ \ &\textrm{if } \ \ \s=1/2,\\
	C \ S^{2m-1}\ \ \ \ &\textrm{if } \ \ \s\in(1/2,1).\\
	\end{cases}
	\end{align*}
	To bound the second integral (note that it only appears in the proof when $1/ \mu \leq S$), we write it in the $(\overline{s},\overline{t})$ variables in $\R^2$, where
	$$
	\overline{s} := \dfrac{|x'|+|x''|}{\sqrt{2}} \, \quad \text{ and } \, \quad  \overline{t} := \dfrac{|x'|-|x''|}{\sqrt{2}}\,.
	$$
	Note that $\overline{t}$ is the signed distance to the cone (see Lemma~4.2 in \cite{CabreTerraI}). Thus, still denoting by $B_S$ the ball of radius $S$ in $\R^2$,
	\begin{align*}
	\int_{B_{S}\cap \{\mu \dist(x,\ccal)>1\}} \dist(x,\ccal)^{-2\s} \d x &\leq C \int \int_{B_{S}\cap \{\overline{s}\geq|\overline{t}|>1/\mu\}} |\overline{t}|^{-2\s} \, (\overline{s}^2-\overline{t}^2)^{m-1} \d \overline{s}\d \overline{t} \\
	& \leq C \int \int_{B_{S}\cap \{\overline{s}\geq|\overline{t}|>1/\mu\}} |\overline{t}|^{-2\s} \, \overline{s}^{2m-2} \d \overline{s}\d \overline{t} \\
	& \leq C\, \int_{1/\mu}^S \d \overline{t}   \ \overline{t}^{-2\s}\int_0^S \d \overline{s}\  \overline{s}^{2m-2} \\
	& \leq \begin{cases}
	C \ S^{2m-2\s}\ \ \ \ &\textrm{if } \ \ \s\in(0,1/2),\\
	C\ S^{2m-1}\,\log S\ \ \ \ &\textrm{if } \ \ \s=1/2,\\
	C \ S^{2m-1}\ \ \ \ &\textrm{if } \ \ \s\in(1/2,1).\\
	\end{cases}
	\end{align*}
	
	Using these two estimates, combined with \eqref{Eq:EnergyEstimateProofLastEstimate} and \eqref{Eq:EnergyEstimatedS},  the desired result follows by noticing that the term $C S^{2m-1}$ in \eqref{Eq:EnergyEstimateProofLastEstimate} is of lower order when $\s \leq 1/2$.
\end{proof}

\section{Existence of saddle-shaped solutions: variational method}
\label{Sec:Existence}

In this section we establish the existence of saddle-shaped solutions to the integro-differential Allen-Cahn equation. The proof is based on the direct method of the calculus of variations, and it uses most of the results appearing in the previous sections.

\begin{proof}[Proof of Theorem~\ref{Th:Existence}]
	Since $\ecal(w,B_R)$ is bounded from below ---by $0$---, we can take a minimizing sequence in $\widetilde{\H}^K_{0, \,\mathrm{odd}}(B_R)$, that we call $u_R^j$ with $j\in \Z^+$. Note that, by Lemma~\ref{Lemma:DecreaseEnergy} we can assume that $-1 \leq u_R^j \leq 1$ and that $u_R^j \geq 0$ in $\ocal$ and  $u_R^j \leq 0$ in $\ical$. 
	
	Now, using $\eqref{Eq:Ellipticity}$, $G\geq 0$, and the fact that $u_R^j$ is a minimizing sequence, we deduce using \eqref{Eq:Energy} that 
	$$
	[u_R^j]^2_{H^\s(B_R)} \leq \dfrac{c_{n,\s}}{\lambda}  [u_R^j]^2_{\H^K(B_R)}\leq \dfrac{2 c_{n,\s}}{\lambda}\ecal(u_R^j,B_R) \leq C
	$$
	for a constant $C$ that does not depend on $j$. Therefore, by combining this with the fractional Poincaré inequality (recall that $u_R^j \equiv 0$ in $\R^{2m}\setminus B_R$) we get that the sequence $\{u_R^j\}$ is bounded in $H^\s(B_R)$. Hence, by the compact embedding $H^\s(B_R) \subset \subset L^2(B_R)$ (see \cite{Adams} and Theorem~7.1 of \cite{HitchhikerGuide}), there exists a subsequence, still denoted by $u_R^j$,  that converges to some doubly radial $u_R \in L^2(B_R)$, and thus, a.e. in $B_R$. By Fatou's lemma, we have
	$$
	\ecal(u_R, B_R)
	\leq \liminf_{j\to \infty} \ecal(u_R^j, B_R) = \inf \setcond{\ecal(w, B_R)}{w \in \widetilde{\H}^K_{0, \,\mathrm{odd}}(B_R)}.
	$$
	Therefore, $u_R \in \widetilde{\H}^K(B_R)$. In addition, $u_R(x) = - u_R(x^\star)$ for every $x\in \R^{2m}$, and $u_R \equiv 0 $ in $\R^{2m} \setminus B_R$. Thus, $u_R$ is a minimizer of $\ecal(\cdot, B_R)$ in $\widetilde{\H}^K_{0, \,\mathrm{odd}}(B_R)$. Moreover, it satisfies $-1\leq u_R \leq 1$ in $B_R$ and $u_R\geq 0$ in $\ocal$. As a consequence, by Proposition~\ref{Prop:WeakSolutionForAllTestFunctions} and the regularity for operators in $L_K$ (see Remark~\ref{Remark:InteriorRegularity}), we have that $u_R$ is a classical solution to
	$$
	 \beqc{\PDEsystem}
	 L_K  u_R &=& f(u_R) & \textrm{ in } B_R\,,\\
	 u_R &=& 0 & \textrm{ in }\R^{2m} \setminus B_R.
	 \eeqc
	$$	
	
	The next step is to pass to the limit in $R$ to obtain a solution in $\R^{2m}$. This is done using a compactness argument. Let $S>0$ and consider the family $\{u_R\}$, for $R> S + 1$, of solutions to $L_K u_R = f(u_R)$ in $B_{S}$. Note first that, if $w$ solves $L_K w = f(w)$ in $B_\rho$ and  $|w|\leq 1$ in $\R^{2m}$ with $f\in C^{\alpha}([-1,1])$ for some $\alpha > 0$, the combination of the estimates \eqref{Eq:C2sEstimate} and \eqref{Eq:Calpha->Calpha+2sEstimateBalls} yields
	$$
	\norm{w}_{C^{2\s + \varepsilon}(B_{\rho/8})} \leq C \bpar{n,\ \s ,\ \lambda,\ \Lambda ,\ \norm{f}_{C^{\alpha}([-1,1])} }\,.
	$$
	for some $\varepsilon > 0$.  By applying this to $u_R$ in balls of radius $\rho = 1$ and centered at points in $\overline{B_{S}}$, we obtain a uniform $C^{2\s + \varepsilon}(\overline{B_S})$ bound for $u_R$. By the Arzelà-Ascoli theorem, as $R\to +\infty$, a subsequence of $\{u_R\}$ converges in $C^{2\s + \varepsilon/2}(\overline{B_S})$ to a (pointwise) solution in $B_S$. Taking now $S = 1,2,3,\ldots$ and using a diagonal argument, we obtain a sequence $u_{R_j}$ converging uniformly on compacts in the $C^{2\s + \varepsilon/2}$ norm to a solution $u \in C^{2\s + \varepsilon/2}(\R^{2m})$ of \eqref{Eq:NonlocalAllenCahn}.
	
	Therefore, we have obtained a solution $u$ to $L_K u = f(u)$ in $\R^{2m}$ which is doubly radial. Furthermore, $u$ is odd with respect to the Simons cone $\ccal$, i.e., $u(x) = -u(x^\star)$ for $x\in \R^{2m}$, and $0 \leq u\leq 1$ in $\ocal$.
	
	Finally, we show that $0<u<1$ in $\ocal$. This will ensure that $u$ is a saddle-shaped solution. First, note that $|u| < 1$ by the usual strong maximum principle (since $u$ vanishes at $\ccal$ and is continuous, we have $u \not \equiv 1$  and $u\not\equiv -1$ in $\R^{2m}$). Let us show now that $u\not\equiv 0$. To do this, we use the energy estimate of Theorem~\ref{Th:EnergyEstimate}. That is, if we consider $u_R$ the minimizer of $\ecal(\cdot, B_R)$ with $R > 8$, we have
	$$
	\ecal (u_R,B_S) \leq \begin{cases}
	C \ S^{2m-2\s}\ \ \ \ &\textrm{if } \ \ \s\in(0,1/2),\\
	C\ S^{2m-2\s} \log S \ \ \ \ &\textrm{if } \ \ \s=1/2,\\
	C \ S^{2m-1}\ \ \ \ &\textrm{if } \ \ \s\in(1/2,1),\\
	\end{cases} $$
	for every $2 < S < R-5$ and with a constant $C$ independent of $R$ and $S$. By letting $R \to \infty$ we obtain the same estimate for $u$. By contradiction, assume $u\equiv 0$. Then, the previous estimate leads to
	$$
	c_m G(0)S^{2m} = \ecal(0, B_S) \leq \begin{cases}
	C \ S^{2m-2\s}\ \ \ \ &\textrm{if } \ \ \s\in(0,1/2),\\
	C\ S^{2m-2\s} \log S\ \ \ \ &\textrm{if } \ \ \s=1/2,\\
	C \ S^{2m-1}\ \ \ \ &\textrm{if } \ \ \s\in(1/2,1),\\
	\end{cases} $$
	and, since $G(0)>0$, this is a contradiction for $S$ large enough. Therefore, $u \not \equiv 0$ and the strong maximum principle for odd functions (see Proposition~\ref{Prop:MaximumPrincipleForOddFunctions}) yields that $u>0$ in $\ocal$. 
\end{proof}

\appendix

\section{Some auxiliary results on convex functions}
\label{Sec:AuxiliaryResults}

In this appendix we present some auxiliary results concerning convex functions. The main result, used in the proof of Theorem~\ref{Th:SufficientNecessaryConditions}, is the following.

\begin{proposition}
	\label{Prop:EquivalenceK(sqrt)Convex<->Inequality}
	Let $K:(0, +\infty) \to (0,+\infty)$ be a measurable function. Then, the following statements are equivalent:
	\begin{enumerate}
		\item[i)] $K(\sqrt{\cdot})$ is strictly convex in $(0, +\infty)$.
		\item[ii)] For every positive constants $c_1$ and $c_2$, the function $g:(0,1/c_2)\to \R$ defined by
		\begin{equation}
		\label{Eq:DefinitiongFromK}
		g(z) := K(c_1 \sqrt{1 + c_2z}) + K(c_1 \sqrt{1 - c_2z})
		\end{equation}
		satisfies 
		\begin{equation}
		\label{Eq:InequalityConvexFunctions}
		g(A) + g(D) \geq g(B) + g(C)
		\end{equation}
		whenever $A$, $B$, $C$, and $D$ belong to $(0, 1/c_2)$ and satisfy
		$$
		A = \max\{A,\, B,\, C,\, D\} \quad \text{ and } \quad A + D \geq B + C.
		$$
		
		In addition, still assuming $A = \max\{A,\, B,\, C,\, D\}$ and $A + D \geq B + C$, equality holds in \eqref{Eq:InequalityConvexFunctions} if and only if the sets $\{A,D\}$ and $\{B,C\}$ coincide.	
	\end{enumerate}
\end{proposition}

To prove this proposition, we need a lemma on convex functions.

\begin{lemma}
	\label{Lemma:ConvexFunctions}
	Let $0<M\leq +\infty$ and let $h:(0,M)\to \R$ be a measurable nondecreasing function. Then, the following statements are equivalent.
	
	\begin{enumerate}[label=(\alph*)]
		\item $h$ is convex in $(0,M)$.
		
		\item For every $0\leq L\leq 2M$, the function $h_L (x) := h(x) + h(L-x)$ is convex in $(\max \{L-M,0\}, \min \{L,M\})$.
		
		\item For every $A$, $B$, $C$, $D$ in the interval $(0,M)$ such that
		$$
		A = \max\{A,\, B,\, C,\, D\} \quad \text{ and } \quad A + D \geq B + C\,,
		$$
		it holds
		\begin{equation}
		\label{Eq:InequalityConvexFunctionsbis}
		h(A) + h(D) \geq h(B) + h(C)\,.
		\end{equation}
	\end{enumerate}
\end{lemma}

\begin{proof}
	$(a)\Rightarrow (c)$.	Since $B$ and $C$ are interchangeable and $h$ is nondecreasing, we may assume that $A \geq B \geq C \geq D$. Now, let $M_C$ be the maximum slope of the supporting lines of $h$ at $C$, and let $m_B$ be the minimum slope of the supporting lines of $h$ at $B$. By the convexity and monotonicity of $h$, it holds $m_B \geq M_C\geq 0$ and also
	$$
	h(x) \geq h(B) + m_B (x-B) \quad \text{ and } \quad h(x) \geq h(C) + M_C (x-C) 
	$$
	for every $x \in (0,M)$.
	
	Hence, since $A-B \geq C-D\geq 0$, we have
	$$
	h(A)-h(B) \geq m_B(A-B) \geq M_C (C-D) \geq h(C) - h(D)\,.
	$$
	
	$(c) \Rightarrow (b)$. Let $x$, $y\in (\max \{L-M,0\}, \min \{L,M\})$ and assume that $x>y$. By taking $A=x$, $B=C=(x+y)/2$, and $D = y$ in \eqref{Eq:InequalityConvexFunctionsbis}, we get 
	$$
	\dfrac{h(x) + h(y)}{2} \geq h \left( \dfrac{x+y}{2}\right). 
	$$ 
	Similarly, by taking $A= L-y$, $B=C=L -(x+y)/2$, and $D = L-x$ in \eqref{Eq:InequalityConvexFunctions}, we get 
	$$
	\dfrac{h(L-x) + h(L-y)}{2} \geq h \left(L - \dfrac{x+y}{2}\right). 
	$$
	By adding up the previous two inequalities we obtain
	$$
	\dfrac{h_L(x) + h_L(y)}{2} \geq h_L \left( \dfrac{x+y}{2}\right). 
	$$ 
	
	$(b) \Rightarrow (a)$. Let $x_0$, $y_0 \in (0,M)$ and choose $L = x_0 + y_0 \leq 2M$. By $(b)$ we have 
	$$
	\dfrac{h(x) + h(x_0 + y_0-x) + h(y) + h(x_0 + y_0-y)}{2} \geq h \left( \dfrac{x+y}{2}\right) + h \left(x_0 + y_0 - \dfrac{x+y}{2}\right),
	$$
	for every $x$ and $y$ in the interval $(\max \{L-M,0\}, \min \{L,M\})$. By choosing $x=x_0$ and $y=y_0$ we obtain
	$$
	h(x_0) + h(y_0)\geq 2 h \left( \dfrac{x_0+y_0}{2}\right). 
	$$
\end{proof}

\begin{remark}
	\label{Remark:StrictConvexity}
	We can replace convexity by strict convexity in $(a)$ and $(b)$, and then the inequality in \eqref{Eq:InequalityConvexFunctionsbis} is strict unless the sets $\{A,D\}$ and $\{B,C\}$ coincide.
\end{remark}

\begin{remark}
	\label{Remark:h_Lincreasing}
	Note that the function $h_L$ is even with respect to $L/2$. Thus, if it is convex, it is nondecreasing in $(L/2, \min \{L,M\})$.
\end{remark}

\begin{remark}
	\label{Remark:hypothesisNondecreasing}
	The assumption of $h$ being nondecreasing is only used to deduce $(c)$ from $(a)$. It is not required to show the equivalence between $(a)$ and $(b)$, neither to deduce $(a)$ from $(c)$.
\end{remark}

With this result available we can show now Proposition~\ref{Prop:EquivalenceK(sqrt)Convex<->Inequality}

\begin{proof}
	$i) \Rightarrow ii)$ We take $M = +\infty$ and $h(\cdot) = K(\sqrt{\cdot})$ in Lemma~\ref{Lemma:ConvexFunctions}. Since $h$ is strictly convex, the function $h_L$ is strictly convex in $(0,L)$ for every $L> 0$ (recall that we do not need to assume that $h$ is monotone to deduce this, see Remark~\ref{Remark:hypothesisNondecreasing}). Moreover, by Remark~\ref{Remark:h_Lincreasing}, $h_L$ is nondecreasing in $(L/2,L)$. Thus, the function $\phi(\cdot) = h_L(\cdot + L/2)$ is strictly convex in $(-L/2,L/2)$ and nondecreasing in $(0,L/2)$. If we choose $L=2c_1^2$, we have that $\phi((L/2)c_2 \cdot) = g(\cdot)$, where $g$ is defined by \eqref{Eq:DefinitiongFromK}. Therefore, $g$ is strictly convex in $(-1/c_2, 1/c_2)$ and nondecreasing in $(0,1/c_2)$. Thus, the result follows by applying  Lemma~\ref{Lemma:ConvexFunctions} to $g$ in $(0,1/c_2)$ (taking into account Remark~\ref{Remark:StrictConvexity}).

	$ii) \Rightarrow i)$ By Lemma~\ref{Lemma:ConvexFunctions} applied to $g$ we deduce that $g$ is strictly convex and nondecreasing in $(0,1/c_2)$ ---take $C=D$ to see that $g$ is monotone. Thus, since $g$ is even and nondecreasing, $g$ is strictly convex in $(-1/c_2,1/c_2)$ and $\varphi(\cdot) = g(\cdot/(c_1^2 c_2))$ is strictly convex in $(-c_1^2, c_1^2)$. Hence, if we call $h(\cdot) := K(\sqrt{\cdot})$ and $L:= 2c_1^2$, we have that $\varphi(\cdot - c_1^2) = h(\cdot) + h(L-\cdot) =:  h_L(\cdot)$, and thus $h_L$ is strictly convex in $(0,L)$. Note that since $c_1>0$ is arbitrary, $h_L$ is strictly convex in $(0,L)$ for all $L>0$. Therefore, by Lemma~\ref{Lemma:ConvexFunctions}, with  $M = +\infty$, we conclude that $h(\cdot) = K(\sqrt{\cdot})$ is strictly convex in $(0,+\infty)$.
\end{proof}

\section{An auxiliary computation}
\label{Sec:AuxiliaryResults2}

In this appendix we present an auxiliary computation that is needed in Section~\ref{Sec:OperatorOddF} in order to complete the proof of Proposition~\ref{Prop:KernelInequalitySufficientCondition}.

\begin{lemma}
	\label{Lemma:ComputationABCD} Let $\alpha$, $\beta$ be two real numbers satisfying $\alpha \geq
	|\beta|$. Let $x=(x',x'')$, $y=(y',y'')\in \ocal \subset \R^{2m}$. Define
	$$
	\begin{array}{cc}
	A = |x'||y'|  \alpha + |x''||y''|\beta \,, \ \ \ \ \ &
	B = |x'||y''| \alpha + |x''||y'| \beta \,, \\
	C = |x''||y'| \alpha + |x'||y''| \beta \,, \ \ \ \ \ &
	D = |x''||y''|\alpha + |x'||y'|  \beta \,.
	\end{array}
	$$
	Then,
	\begin{enumerate}
		\item It holds
		$$
		\begin{cases}
		|A| \geq |B|,\ |A| \geq|C|, \ |A| \geq|D|\,, \\
		|A| + |D| \geq |B| + |C|\,.
		\end{cases}
		$$
		\item If  the sets $\{|A|,|D|\}$ and $\{|B|,|C|\}$ coincide, then necessarily $\alpha = \beta = 0$.
	\end{enumerate}
	
\end{lemma}
\begin{proof} The proof is elementary but requires to check some cases. In all of them we will use the following inequalities. Since $\alpha \geq |\beta |$,
	$$
	\alpha\geq 0 \quad \textrm{ and } \quad  -\alpha \leq \beta \leq \alpha\,.
	$$
	Moreover, since $x,y\in\ocal$, it holds
	$$
	|x'|>|x''| \quad \textrm{ and } \quad |y'|>|y''|\,.
	$$

	We start establishing the first statement. We show next that $A\geq 0$ and that
	$$
	A \geq |B|, \ A \geq |C| ,\ A \geq |D|\,.
	$$

	$\bullet$ $A \geq 0$:
	$$
	A =  |x'||y'|  \alpha + |x''||y''|\beta \geq (|x'||y'|  - |x''||y''|)\alpha \geq 0\,.
	$$
	
	$\bullet$ $A \geq |B|$:
	$$
	A\pm B = (|x'|\alpha-|x''|\beta)(|y'|\pm |y''|) \geq 0\,.
	$$
	
	$\bullet$ $A \geq |C|$:
	$$
	A\pm C = (|y'|\alpha-|y''|\beta)(|x'|\pm |x''|)  \geq 0\,.
	$$
	
	$\bullet$ $A \geq |D|$:
	$$
	A\pm D = (|x'||y'| \pm |x''||y''|)(\alpha \pm \beta) \geq 0\,.
	$$

	It remains to show
	$$
	A + |D| \geq |B| + |C|\,.
	$$
	The proof of this fact is just a computation considering all the eight possible configurations of the signs of $B$, $C$, and $D$. Since the roles of $B$ and $C$ are completely interchangeable, we may assume that $B \geq C$ and we only need to check six cases. To do it, note first that
	\begin{equation}
	\label{Eq:LemmaABCDProof1}
	A + D - B - C = (|x'|-|x''|)(|y'|-|y''|)(\alpha + \beta) \geq 0 \,,
	\end{equation}
	\begin{equation}
	\label{Eq:LemmaABCDProof2}
	A - D - B + C = (|x'|+|x''|)(|y'|-|y''|)(\alpha - \beta) \geq 0 \,,
	\end{equation}
	and
	\begin{equation}
	\label{Eq:LemmaABCDProof3}
	A + D + B + C = (|x'|+|x''|)(|y'|+|y''|)(\alpha + \beta) \geq 0 \,,
	\end{equation}
	With these three relations at hand we check the six cases.
	
	$\bullet$ If $B \geq 0$, $C \geq 0$, and $D \geq 0$, then by \eqref{Eq:LemmaABCDProof1} we have
	$$
	A + |D| - |B| - |C| = A + D - B - C \geq 0\,.
	$$
	
	$\bullet$ If $B \geq 0$, $C \geq 0$, and $D \leq 0$, we use the sign of $D$ and \eqref{Eq:LemmaABCDProof1} to see that
	$$
	A + |D| - |B| - |C| = A - D - B - C =  (A + D - B - C) + (-2D) \geq 0\,.
	$$
	
	$\bullet$ If $B \geq 0$, $C \leq 0$, and $D \geq 0$, we use the sign of $D$ and \eqref{Eq:LemmaABCDProof2} to see that
	$$
	A + |D| - |B| - |C| = A + D - B + C =  (A - D - B + C) + 2D \geq 0\,.
	$$
	
	$\bullet$ If $B \geq 0$, $C \leq 0$, and $D \leq 0$, then by \eqref{Eq:LemmaABCDProof2} we have
	$$
	A + |D| - |B| - |C| = A - D - B + C \geq 0\,.
	$$
	
	$\bullet$ If $B \leq 0$, $C \leq 0$, and $D \geq 0$, then by \eqref{Eq:LemmaABCDProof3} we have
	$$
	A + |D| - |B| - |C| = A + D + B + C \geq 0\,.
	$$
	
	$\bullet$ If $B \leq 0$, $C \leq 0$, and $D \leq 0$, we use the sign of $D$ and \eqref{Eq:LemmaABCDProof3} to see that
	$$
	A + |D| - |B| - |C| = A - D + B + C =  (A + D + B + C) + (-2D) \geq 0\,.
	$$
	
	This concludes the proof of the first statement.
	
	We prove now the second point of the lemma. Since the roles of $B$ and $C$ are completely interchangeable, we only need to show the result in the case $|A| = |B|$ and $|C| = |D|$.
	
	Recall that $A \geq 0$. Hence, since $A = |B|$ and $|C| = |D|$, a simple computation shows that
	$$
	\alpha = \sign (B) \dfrac{|x''|}{|x'|}\beta \quad \textrm{ and } \quad
	\beta = \sign (C) \sign(D) \dfrac{|x''|}{|x'|} \alpha \,.
	$$
	Hence, combining both equalities we obtain
	$$
	\alpha = \sign (B) \sign (C) \sign(D) \dfrac{|x''|^2}{|x'|^2} \alpha.
	$$
	Finally, if we assume $\alpha \neq 0$, then necessarily $\sign (B) \sign (C) \sign(D)=1$ and $|x'|= |x''|$, but this is a contradiction with $x\in \ocal$. Therefore, $\alpha = 0$ and thus $\beta =0$.
\end{proof}

\section{The integro-differential operator $L_K$ in the $(s,t)$ variables}
\label{Sec:stcomputations}

The goal of this appendix is to take advantage of the doubly radial symmetry of the functions we are dealing with to write equation \eqref{Eq:NonlocalAllenCahn} in $(s,t)$ variables, passing from an equation in $\R^{2m}$ to an equation in $(0,+\infty)\times (0,+\infty)\subset \R^2$. Although we do not use these computations in this paper, we include them here to show the usefulness of having introduced the $\overline{K}$ kernel obtained after integration with respect to the Haar measure on $O(m)^2$. Moreover, the following expressions could be useful for future reference. In the case of the fractional Laplacian, the kernel that we obtain involves essentially an hypergeometric function which is the so-called Appell function $F_2$ (see \cite{Appell} for its definition).

\begin{lemma}
	\label{Lemma:OperatorInSTVariables} Let $m \geq 1$, $\s\in(0,1)$, and let $w\in
	C^\alpha(\R^{2m})$, with $\alpha > 2\s$, be a doubly radial function, i.e., depending only on the variables $s$ and $t$. Let $L_K$ be a rotation invariant operator, that is, $K(z) = K(|z|)$, of the form \eqref{Eq:DefOfLu}. Then, if we define $\tilde{w}:(0,+\infty)\times (0,+\infty) \to \R$ by $\tilde{w}(s,t) = w(s,0,...,0,t,0,...,0)$, it holds
	$$ L_Kw(x) = \tilde{L}_K \tilde{w} (|x'|,|x''|), $$
	with
	\begin{equation*}
	\label{Eq:OperatorInSTVariables}
	\widetilde{L}_K \tilde{w} (s,t) := \int_0^{+\infty}  \int_0^{+\infty} \sigma^{m-1} \tau^{m-1} \big(\tilde{w}(s,t) - \tilde{w}(\sigma, \tau)\big) J(s,t,\sigma, \tau)  \d \sigma\d \tau\,,
	\end{equation*}
	where:
	\begin{enumerate}
		\item If $m= 1$,
		\begin{equation}
		\label{Eq:KernelSTVariables1}
		J(s,t,\sigma, \tau) := \sum_{i=0}^1  \sum_{j =0}^1  K\Big(\sqrt{s^2 + t^2 + \sigma^2 + \tau^2 -2 s \sigma (-1)^i -2 t \tau (-1)^j}\Big)\,.
		\end{equation}
		
		\item If $m\geq 2$,
		\begin{align}
		J(s,t,\sigma, \tau) &:= c_m ^2  \int_{-1}^1  \int_{-1}^1  (1-\theta^2)^{\frac{m-2}{2}} (1-\overline{\theta}^2)^{\frac{m-2}{2}} \nonumber\\
		& \quad \quad \quad \quad \quad
		K\Big(\sqrt{s^2 + t^2 + \sigma^2 + \tau^2 -2 s \sigma \theta -2 t \tau \overline{\theta}}\Big) \d \theta \d \overline{\theta}\,, \label{Eq:KernelSTVariables2}
		\end{align}
		with
		$$
		c_m = \dfrac{2 \pi^{\frac{m-1}{2}}}{\Gamma (\frac{m-1}{2})}.
		$$
	\end{enumerate}
\end{lemma}

\begin{proof}
	Let $x = (s x_s, t x_t)$ with $x_s$, $x_t$ $\in \Sph^{m-1}$ and $y = (\sigma y_\sigma, \tau
	y_\tau)$ with $y_\sigma$, $y_\tau$ $\in \Sph^{m-1}$. Then, decomposing $\R^{2m} = \R^m \times \R^m$
	and using spherical coordinates in each $\R^m$ we obtain
	\begin{align*}
	L_Ku(x) &= \int_{\R^{2m}} \big( u(x) - u(y)\big) K( |x-y|) \d y &\\
	&= \int_0^{+\infty}  \int_0^{+\infty} \sigma^{m-1} \tau^{m-1} \big(u(s,t) - u(\sigma, \tau)\big)  \\
	&\quad \quad \quad \quad  \bpar{\int_{\Sph^{m-1}}  \int_{\Sph^{m-1}} K \Big( \sqrt{|sx_s - \sigma y_\sigma|^2 + |t x_t - \tau y_\tau|^2 } \Big) \d y_\sigma \d y_\tau } \d \sigma \d \tau .
	\end{align*}
	Now, we define the kernel
	\begin{equation}
	\label{Eq:KernelSTVariablesProof1}
	J(x_s, x_t, s,t,\sigma, \tau) := \int_{\Sph^{m-1}}  \int_{\Sph^{m-1}} K \Big( \sqrt{|sx_s - \sigma y_\sigma|^2 + |t x_t - \tau y_\tau|^2 }\Big ) \d y_\sigma \d y_\tau \,.
	\end{equation}
	
	First of all, it is easy to see that $J$ does not depend on $x_s$ nor $x_t$. Indeed, consider a different point $(z_s, z_t)\in \Sph^{m-1} \times \Sph^{m-1}$ and let $M_s$ and $M_t$ be two orthogonal transformations such that $M_s(x_s) = z_s$ and $M_t(x_t) = z_t$. Then, making the change of variables $y_\sigma = M_s(\tilde{y}_\sigma)$ and $y_\tau = M_t(\tilde{y}_\tau)$, and using that $M_s( \Sph^{m-1}) = M_t(\Sph^{m-1}) = \Sph^{m-1}$, we find out that
	\begin{align*}
	& \hspace{-1cm} J(z_s, z_t, s,t,\sigma, \tau) = \\
	&= \int_{\Sph^{m-1}}  \int_{\Sph^{m-1}} K \Big( \sqrt{|s M_s(x_s) - \sigma y_\sigma|^2 + |t M_t(x_t) - \tau y_\tau|^2 }\Big) \d y_\sigma \d y_\tau \\
	&= \int_{\Sph^{m-1}}  \int_{\Sph^{m-1}} K \Big( \sqrt{|s M_s(x_s) - \sigma M_s(\tilde{y}_\sigma)|^2 + |t M_t(x_t) - \tau M_t(\tilde{y}_\tau)|^2 }\Big) \d \tilde{y}_\sigma \d \tilde{y}_\tau \\
	&= \int_{\Sph^{m-1}}  \int_{\Sph^{m-1}} K\Big ( \sqrt{|M_s(sx_s - \sigma \tilde{y}_\sigma)|^2 + |M_t(t x_t - \tau \tilde{y}_\tau)|^2 }\Big) \d \tilde{y}_\sigma \d \tilde{y}_\tau \\
	&= \int_{\Sph^{m-1}}  \int_{\Sph^{m-1}} K\Big ( \sqrt{|sx_s - \sigma \tilde{y}_\sigma|^2 + |t x_t - \tau \tilde{y}_\tau|^2 }\Big) \d \tilde{y}_\sigma \d \tilde{y}_\tau \\
	&= J(x_s, x_t, s,t,\sigma, \tau) \,.
	\end{align*}
	
	Therefore, we can replace $x_s$ and $x_t$ in \eqref{Eq:KernelSTVariablesProof1} by $e =(1,0,\ldots,0) \in \Sph^{m-1}$. Thus, we have
	\begin{equation*}
	J(s,t,\sigma, \tau) := \int_{\Sph^{m-1}}  \int_{\Sph^{m-1}} K\Big( \sqrt{|s e - \sigma y_\sigma|^2 + |t e - \tau y_\tau|^2 }\Big) \d y_\sigma \d y_\tau \,.
	\end{equation*}
	For an easier notation, we rename $\omega = y_\sigma$ and $\tilde\omega = y_\tau$, and thus we have
	\begin{align*}
	|s e - \sigma y_\sigma|^2 + |t e - \tau y_\tau|^2 &= |s e - \sigma \omega|^2 + |t e - \tau \tilde\omega|^2\\
	&= s^2 +\sigma^2 - 2 s \sigma e \cdot \omega + t^2 + \tau^2 - 2 t \tau e\cdot \tilde\omega \\
	&= s^2 +\sigma^2 - 2 s \sigma \omega_1 + t^2 + \tau^2 - 2t \tau\tilde\omega_1\,.
	\end{align*}
	Then, we can rewrite $J$ as
	\begin{equation*}
	\label{Eq:KernelSTVariablesProof3}
	J(s,t,\sigma, \tau) := \int_{\Sph^{m-1}}  \int_{\Sph^{m-1}} K\Big( \sqrt{s^2+\sigma^2- 2 s \sigma \omega_1 + t^2 + \tau^2 - 2t \tau\tilde\omega_1}\Big) \d \omega \d \tilde\omega \,.
	\end{equation*}
	At this point we have to distinguish the cases $m=1$ and $m\geq 2$. For the first one, since
	$\Sph^{0} = \{-1,1\}$ we directly obtain \eqref{Eq:KernelSTVariables1}. For the second one,
	since the integrand only depends on $\omega_1$ and $\tilde\omega_1$, defining $\rho(\cdot) = \sqrt{1-|\cdot|^2}$ we proceed as follows
	\begin{align*}
	\label{Eq:KernelSTVariablesProof4}
	J(s,t,\sigma, \tau) &= \int_{\Sph^{m-1}}  \int_{\Sph^{m-1}} K\Big( \sqrt{s^2+\sigma^2- 2 s \sigma \omega_1 + t^2 + \tau^2 - 2t \tau\tilde\omega_1}\Big) \d \omega \d \tilde\omega \,\\
	&= \int_{-1}^1 \d \omega_1 \int_{\partial B_{\rho(\omega_1)}} \d \omega_2\cdot\cdot\cdot\d \omega_m \int_{-1}^1 \d \tilde\omega_1 \int_{\partial B_{\rho(\tilde\omega_1)}} \d \tilde\omega_2\cdot\cdot\cdot\d \tilde\omega_m  \\
	& \quad \quad \quad \quad \quad K\Big( \sqrt{s^2+\sigma^2- 2 s \sigma \omega_1 + t^2 + \tau^2 - 2t \tau\tilde\omega_1}\Big) \, \\
	&= \int_{-1}^1 \int_{-1}^1  |\partial B_{\rho(\omega_1)}| |\partial B_{\rho(\tilde\omega_1)}|\,\\
	& \quad \quad \quad \quad \quad K\Big( \sqrt{s^2+\sigma^2- 2 s \sigma \omega_1 + t^2 + \tau^2 - 2t \tau\tilde\omega_1}\Big) \d \omega_1 \d \tilde\omega_1 .
	\end{align*}
	Finally, we obtain \eqref{Eq:KernelSTVariables2} once we replace
	$|\partial B_{\rho(\cdot)}|=c_m\,\rho(\cdot)^{m-2} = c_m (1-|\cdot|^2)^\frac{m-2}{2}$, where $c_m$ is the measure of the boundary of the unitary ball in $\R^{m-1}$.
\end{proof}

If the operator $L_K$ is the fractional Laplacian, the previous expression of the kernel $J$ can be rewritten in terms of a hypergeometric function of two variables, the so-called Appell function $F_2$ (see \cite{Appell}). This expression does not simplify any of the arguments of this paper. Nevertheless, we think that it is worthy to point out the relation between $J$ and $F_2$, since the known properties of the last one could provide some information about the kernel $J$.

\begin{lemma}
	\label{Lemma:Appell} Let $F_2$ be the Appell hypergeometric function defined in \cite{Appell}. If $L_K = (-\Delta)^\s$ and $m\geq 2$, then
	\begin{equation}
	\label{Eq:Appell}
	J(s,t,\sigma,\tau) = \frac{c_{2m,\s}\pi^m\Gamma\left(\frac{m}{2}\right)^2}{\Gamma\left(\frac{m-1}{2}\right)^2\Gamma\left(\frac{m+1}{2}\right)^2} \frac{F_2\left( m+\s;\frac{m}{2},m;\frac{m}{2},m;\frac{4s\sigma}{(s+\sigma)^2+(t+\tau)^2},\frac{4t\tau}{(s+\sigma)^2+(t+\tau)^2} \right)}{[(s+\sigma)^2+(t+\tau)^2]^{m+\s}}.
	\end{equation}
\end{lemma}

\begin{proof}
	If we take $K(z) = c_{2m,\s}|z|^{-2m-2\s}$ in \eqref{Eq:KernelSTVariables2} we get
	\begin{align*}
	J(s,t,\sigma, \tau) = c_{2m,\s} c_m ^2  \int_{-1}^1  \int_{-1}^1  \frac{(1-\theta^2)^{\frac{m-2}{2}} (1-\overline{\theta}^2)^{\frac{m-2}{2}}}{(s^2 + t^2 + \sigma^2 + \tau^2 -2 s \sigma \theta -2 t \tau \overline{\theta})^{m+\s}} \d \theta \d \overline{\theta}\,.
	\end{align*}
	Then, if we make the change of variables $\theta = 2\varpi_1-1$ and $\overline{\theta}=2\varpi_2-1$
	we arrive at
	\begin{align*}
	J(s,t,\sigma, \tau) &= \frac{ c_{2m,\s} 2^{2m-4} c_m^2}{[(s+\sigma)^2+(t+\tau)^2]^{m+\s}} \cdot \\
	& \quad \quad \quad \quad  \int_0^1 \int_0^1
	\frac{ \varpi_1^\frac{m-2}{2} (1-\varpi_1)^\frac{m-2}{2} \varpi_2^\frac{m-2}{2}
		(1-\varpi_2)^\frac{m-2}{2}}{\left(1-\frac{4s\sigma}{(s+\sigma)^2+(t+\tau)^2}\,\varpi_1-\frac{4t\tau}{(s+\sigma)^2+(t+\tau)^2}\,\varpi_2
		\right)^{m+\s}} \d \varpi_1 \d \varpi_2 \\
	&= \frac{ c_{2m,\s} 2^{2m-4} c_m^2}{[(s+\sigma)^2+(t+\tau)^2]^{m+\s}} \frac{\Gamma\left(\frac{m}{2} \right)^4}{\Gamma(m)^2} \cdot \\
	& \quad \quad \quad \quad F_2\left( m+\s;\frac{m}{2},m;\frac{m}{2},m;\frac{4s\sigma}{(s+\sigma)^2+(t+\tau)^2},\frac{4t\tau}{(s+\sigma)^2+(t+\tau)^2} \right).
	\end{align*}
	We finally obtain \eqref{Eq:Appell} by using the duplication formula for the $\Gamma$-function.
\end{proof}

To conclude the appendix, we rewrite the kernel inequality \eqref{Eq:KernelInequality} in $(s,t)$ variables and in terms of the kernel $J$. We do not present a proof of this result since it is identical to the one of Proposition~\ref{Prop:KernelInequalitySufficientCondition} but changing the notation.

\begin{lemma}
	\label{Lemma:KernelInequalityCone} Let $m\geq 1$ and let $J$ the kernel defined in
	\eqref{Eq:KernelSTVariables1}-\eqref{Eq:KernelSTVariables2} with $K(\sqrt{\cdot})$ strictly convex. Then, if $s>t$ and $\sigma > \tau$, we have
	\begin{equation*}
	J(s,t,\sigma, \tau) > J(s,t,\tau, \sigma)\,.
	\end{equation*}
\end{lemma}

\section*{Acknowledgements}
The authors thank Xavier Cabré for his guidance and useful discussions on the topic of this paper.

\bibliographystyle{amsplain}
\bibliography{biblio}

\end{document}